\documentclass[12pt,leqno]{amsart}
\usepackage{latexsym,amssymb,amscd,amsthm,natbib,amsmath}
\usepackage{setspace}
\newtheorem{thm}{Theorem}[section]
\newtheorem{cor}[thm]{Corollary}
\newtheorem{lem}[thm]{Lemma}
\newtheorem{prop}[thm]{Proposition}

\newtheorem{ques}[thm]{Question}

\theoremstyle{definition}
\newtheorem{ddef}[thm]{Definition}
\newtheorem{ex}[thm]{Example}
\newtheorem{exs}[thm]{Examples}

\newtheorem{rmk}[thm]{Remark}

\theoremstyle{remark}

\newtheorem*{claims}{Claims}

\newcommand{\st}{\ensuremath{\, | \,}}

\newcommand{\Sgen}{S=\langle a_1,...,a_\nu \rangle}
\newcommand{\Rgen}{R=k[[x^{a_1},\dots,x^{a_\nu}]]}

\newcommand{\wt}{\widetilde}
\newcommand{\wh}{\widehat}

\newcommand{\pre}{\preceq}

\newcommand{\floor}[2][a_1]{\left\lfloor\frac{#2}{#1}\right\rfloor}
\newcommand{\ceil}[2][a_1]{\left\lceil\frac{#2}{#1}\right\rceil}
\newcommand{\sfloor}[2][a_1]{\big\lfloor\frac{#2}{#1}\big\rfloor}
\newcommand{\sceil}[2][a_1]{\big\lceil\frac{#2}{#1}\big\rceil}

 \DeclareMathOperator{\Z}{\mathbb Z} \DeclareMathOperator{\N}{\mathbb N}

  \DeclareMathOperator{\A}{\mathcal A}
 \DeclareMathOperator{\I}{\mathcal I} 
  
\DeclareMathOperator{\gr}{\mathrm gr}

\def\sqr#1#2{{\vcenter{\hrule height.#2pt
        \hbox{\vrule width.#2pt height#1pt \kern#1pt
                \vrule width.#2pt}
        \hrule height.#2pt}}}

\def\ds{\displaystyle}

    \DeclareMathOperator{\ord}{ord }
 \DeclareMathOperator{\Ap}{Ap} 
\DeclareMathOperator{\maxap}{maxAp} \DeclareMathOperator{\minap}{minAp}

\begin{document}
\title[Goto Numbers of a Numerical Semigroup Ring]{Goto Numbers of a Numerical Semigroup Ring and the Gorensteiness of Associated Graded Rings}
\author{Lance Bryant}
\date{\today}

\begin{abstract} The Goto number of a parameter ideal $Q$ in a Noetherian local ring $(R,m)$ is the largest integer $q$ such that $Q:m^q$ is integral over $Q$. The Goto numbers of the monomial parameter ideals of $R=k[[x^{a_1},x^{a_2},\dots,x^{a_\nu}]]$ are characterized using the semigroup of $R$. This helps in computing them for classes of numerical semigroup rings, as well as on a case-by-case basis. The minimal Goto number of $R$ and its connection to other invariants is explored. Necessary and sufficient conditions for the associated graded rings of $R$ and $R/x^{a_1}R$ to be Gorenstein are also given, again using the semigroup of $R$.
\end{abstract}

\maketitle
\section{Introduction}

Quasi-socle ideals are ideals of the form $I=Q:m^q$ in a Noetherian local ring $(R,m)$ with dimension $d>0$, where $Q=(a_1,a_2,\dots,a_d)$ is a parameter ideal. In response to a conjecture of \citet{PU} rooted in linkage theory, \citet{W} proved the following:  Let $R$ be Cohen-Macaulay with $d\ge 2$, $d>2$ when $R$ is regular. If $Q\subset m^q$ and $I\ne R$, then $I\subset m^q$, $m^qI = m^qQ$, and $I^2=QI$. Recently \citet{GKM} have inquired about a modification of this result for the one-dimensional case. They began by studying Gorenstein numerical semigroup rings.

One component involved in making this modification is determining when $Q:m^q$ is integral over $Q$. \citet{HS} made the following definition.

\begin{ddef}\label{goto def} Let $g(Q)$ be the largest integer $q$ such that $Q:m^q$ is integral over $Q$. This is called the {\em Goto number of Q}. 
\end{ddef}

\noindent In this paper the Goto numbers of the parameter ideals in a numerical semigroup ring $(R,m)=k[[S]]=k[[x^{a_1},x^{a_2},\dots,x^{a_\nu}]]$ are considered. Of particular interest are the monomial parameter ideals.

Let $e$ be the multiplicity of $R$. In section 2, Corollary \ref{cmain} states that there is an $e$-tuple of integers $(\sigma(1),\sigma(2),\dots,\sigma(e))$ called the Goto vector of $S$ such that $\{g(u) \st 0\ne u\in S\} \subset \{\sigma(1),\sigma(2),\dots,\sigma(e)\}$ where $g(u)$ is the Goto number of $x^uR$. In addition, $\min\{g(Q) \st Q$ is a parameter ideal of $R\} = \min\{\sigma(1),\sigma(2),\dots,\sigma(e)\}$. If the Goto vector of $S$ is known, then the computation of $g(u)$ is reduced to determining the set $\{\alpha \in \{1,2,\dots,e\} \st u-\alpha\in S\}$. Using this set, several results about $g(u)$ are obtained. Bounds for the Goto numbers of $R$ are given in Proposition \ref{bounds}.

Section 3 introduces a class of numerical semigroups called $M$-pure semigroups. For $(R,m)=k[[S]]$, the $M$-purity of $S$ is closely connected to the Gorenstein property of the associated graded rings $\gr_m(R)=\bigoplus_{i\ge0} m^i/m^{i+1}$ and $\gr_{\bar m}(\bar R)=\bigoplus_{i\ge0} \bar{m}^i/\bar{m}^{i+1}$ where $\bar R=R/x^{a_1}R$ and $\bar m = m/x^{a_1}R$. The definition of $M$-purity involves the Ap\'ery set of the semigroup $S$ and the $m$-adic order of $R$. The following equivalencies are proven in Theorem \ref{equivs}: (i) $\gr_{\bar m}(\bar R)$ is Gorenstein if and only if $S$ is $M$-pure symmetric, (ii) $\gr_{m}(R)$ is Gorenstein if and only if $S$ is $M$-pure symmetric and $g(a_1)=r$, where $g(a_1)$ is the Goto number of $x^{a_1}R$ and $r$ is the reduction number of $m$ (with respect to $x^{a_1}R$). This section is ended by showing that in general there are no implications among the conditions symmetry, $M$-purity, and $M$-additivity. A semigroup $S$ is $M$-additive if and only if $\gr_m(R)$ is Cohen-Macaulay.

Section 4 is largely concerned with the inequalities $\delta \le \gamma \le \ord(C) \le \tau \le g(a_1) \le r$ where $\ord(C)$ is the $m$-adic order of the conductor $C=R:k[[x]]$ and $\tau = \min\{g(Q)\st Q$ is a parameter ideal of $R\}$. $\delta$ and $\gamma$ are invariants that are derived from the Ap\'ery set of $S$. Sufficient conditions are provided for many of these inequalities to be equalities. A guiding question for this work has been: when is $\tau = g(a_1)$? Some partial results are the following equivalencies stated in Theorem \ref{puredel}: (i) $S$ is $M$-pure if and only $\delta = g(a_1)$, (ii) $S$ is $M$-pure and $\gr_m(R)$ is Cohen-Macaulay if and only if $\delta = r$, and (iii) $\gr_m(R)$ is Gorenstein if and only if $S$ is symmetric and $\delta = r$. Furthermore $\tau = g(a_1)$ for every semigroup with multiplicity less than 5 with the exception of $S=<4,5,7>$, and for every symmetric semigroup with multiplicity less than 7 with the exceptions $S=<5,6,9>$ and $S=<6,7,10,11>$.

In the last section Goto numbers are expressed in terms of the minimal generators of the semigroup. This enables one to consider a class of semigroups (e.g., those with embedding dimension 2) instead of considering semigroups on a case-by-case basis. The contents of Theorems \ref{arith}, \ref{sym almost}, \ref{maxbed}, \ref{ele4}, and \ref{ee5} are the computations for the following classes: (i) $S$ is symmetric and generated by an arithmetic sequence, (ii) $S$ is symmetric and of almost maximal embedding dimension, (iii) $S$ is of maximal embedding dimension, (iv) $S$ has multiplicity less than 5, (v) $S$ is symmetric with multiplicity equal to 5. Moreover, in Theorem \ref{gu5} the Goto numbers of an arbitrary semigroup $S$ are expressed in terms of the Ap\'ery set of $S$.


\section{The Goto Numbers of a Numerical Semigroup}

A {\em numerical semigroup}, or {\em semigroup}, $S$ is a subsemigroup of $\N_0$ that contains 0 and has a finite complement in $\N_0$. For two elements $u$ and $u'$ in $S$, $u \pre u'$ if there exists an $s\in S$ such that $u+s = u'$. This defines a partial ordering on $S$. The minimal elements in $S\setminus\{0\}$ with respect to this ordering form the unique minimal set of generators for $S$, which is denoted by $\{a_1,a_2,\dots, a_\nu\}$ where $a_1<a_2<\dots<a_\nu$. $S=\{\sum_{i=1}^{\nu} c_ia_i \st c_i\ge 0\}$ is represented using the notation $\Sgen$. Since the minimal generators of $S$ are distinct modulo $a_1$, the set of minimal generators is finite. Furthermore, having finite complement in $\N_0$ is equivalent to $\gcd\{a_i \st 1\le i\le \nu\} = 1$.

$S$ has the following invariants:
\medskip

\begin{itemize}
\item $f=f(S) = \max\{z\in \Z\setminus S\}$ is the {\em Frobenius number} of $S$
\item $e = e(S) = \min\{u\in S\st u\ne 0\}$ is the {\em multiplicity} of $S$
\item $\nu = \nu(S) =\#\{a_1,\dots,a_\nu\}$ is the {\em embedding dimension} of $S$
\end{itemize}
\medskip

Notice that we always have $e=a_1$ and $\nu\le e$.

The {\em Ap\'ery set of $S$ with respect to $n\in \N$} is $\Ap(S;n)=\{w \in S \st w-n\not\in S\}$. This is a finite set and it is immediate from the definition that if $w\in \Ap(S;n)$ and $s\pre w$ for some $s\in S$, then $s\in \Ap(S;n)$. For $u\in S$, $\#\Ap(S;u) = u$ and no two elements of $\Ap(S;u)$ are congruent modulo $u$. $\Ap(S) = \Ap(S;e)$ and is usually called {\em the} Ap\'ery set of $S$. There are two natural ways to order the elements of $\Ap(S)$. One is $\Ap(S)=\{w_0,\dots,w_{e-1}\}$ where $w_0 < w_1 < \cdots < w_{e-1}$, and the other is $\Ap(S)=\{v_0,\dots,v_{e-1}\}$ where $v_n \equiv n \mod e$. Notice that we always have $w_0=v_0=0$.

We will allow the subscript of $v_n$ to be any integer by agreeing that $v_{n} = v_{m}$ if and only if $n \equiv m \mod e$. This will allow us to perform arithmetic operations on the subscripts of the $v_i$'s as in Lemma \ref{apnot}. Recall that the floor function $\left\lfloor x \right\rfloor$ denotes the greatest integer less than or equal to $x$, and that the ceiling function $\left\lceil x \right\rceil$ denotes the least integer greater than or equal to $x$.

\begin{lem}\label{apnot} Let $\Sgen$ be a semigroup with $\Ap(S) = \{v_0,\dots,v_{e-1}\}$. Then
\begin{enumerate}
\item $v_n+v_m = v_{n+m} + ta_1$ where $t\ge 0$.
\item $v_n-v_m = v_{n-m} - ta_1$ where $t\ge 0$.
\item $v_n + v_{-n} = (1+\sfloor{v_n}+\sfloor{v_{-n}})a_1$ if $v_n \ne 0$.
\end{enumerate}
\end{lem}

\begin{proof} (1) Clearly $v_n+v_m = v_{n+m} + ta_1$ for some $t$ since $v_n+v_m \equiv v_{n+m} \mod e$. Moreover, $v_n+v_m \in S$ so $t\ge 0$.

(2) Again $v_n-v_m = v_{n-m} - ta_1$ for some $t$ since $v_n-v_m \equiv v_{n-m} \mod e$. If $t<0$, then $v_n - a_1 \in S$, which is a contradiction.

(3) By (1), $v_n + v_{-n} = v_{0} + ta_1 = ta_1$ where $t\ge 0$. Thus $v_n -a_1\sfloor{v_n} + v_{-n} -a_1\sfloor{v_{-n}}$ is a positive multiple of $a_1$ when $v_n \ne 0$. Since both $v_n -a_1\sfloor{v_n}$ and $v_{-n} -a_1\sfloor{v_{-n}}$ are positive and strictly less than $a_1$, we have $v_n -a_1\sfloor{v_n} + v_{-n} -a_1\sfloor{v_{-n}} = a_1$.
\end{proof}

\noindent Statements (1) and (2) of Lemma \ref{apnot} are also given by \citet[Lemma 2.6]{MCH}. The next definition facilitates the use of both notations for $\Ap(S)$.

\begin{ddef}\label{biject} For $i\in \{0,1,\dots,e-1\}$, let $\wh \imath\in \{1,2,\dots,e\}$ denote the integer such that $w_i = v_{\wh \imath}$.
\end{ddef}

Lemma \ref{apnot} and Definition \ref{biject} will be used extensively in Section 5. However, they are occasionally used throughout the paper and hence recorded here. Notice that $\wh \imath\in \{1,2,\dots,e\}$ instead of $\{0,1,\dots,e-1\}$. This will be useful in Section 5 where we compute $\sigma(\,\wh \imath\,)$ (see Definition \ref{gv gs}).

The {\em numerical semigroup ring} corresponding to $S$ is $R=k[[S]]=k[[x^{a_1},\dots,x^{a_\nu}]]$ where $k$ is a field. This is a one-dimensional local domain with maximal ideal $m=(x^{a_1},\dots,x^{a_\nu})$. The multiplicity and embedding dimension of $S$ coincide with the ring-theoretic notions for the ring $R$. Unless stated otherwise, it is assumed that $R\ne k[[x]]$. For the semigroup $S$, this means that $S \ne \N_0$, or equivalently $\nu\ge 2$.

A parameter ideal $Q$ of $R$ is a principal ideal generated by a nonzero, nonunit element. Although ultimately all of the parameter ideals of $R$ are of interest (and some results with this generality are provided), the focus is on monomial parameter ideals. These are principal ideals generated by a monomial. The benefit is that one can work directly with the numerical semigroup.

For the remainder of this paper $g(u)$ denotes the Goto number of $x^uR$ and $\{g(u) \st 0\ne u\in S\}$ is the set of Goto numbers of $S$. Moreover $\ord(n) = \ord(x^n) = k$ where $x^n\in m^k\setminus m^{k+1}$ if $n\in S$. This is the {\em $m$-adic order of $x^n$}. To simplify statements, we will accept the convention that $\ord(n) = -1$ if $n\not\in S$.

Theorem \ref{main} and Corollary \ref{cmain} are the main results of this section. The former characterizes the Goto number of a nonzero element of $S$ using the Ap\'ery sets $\Ap(S;\alpha)$ where $1\le \alpha\le e$, and the latter introduces the Goto vector of $S$. We begin by defining the following sets.

\begin{ddef}\
\begin{enumerate}
\item $T = \{z \in \Z\setminus S \st z+u \in S \text{ for all } 0\ne u\in S\}$
\item $\A(u) = \{\alpha\in \{1,2,\dots,e\} \st u-\alpha \in S\}$
\item $\A(S) = \{1,2,\dots,e\}$
\end{enumerate}
\end{ddef}
\noindent Results about $T$ can be found in \citep{FGH} and \citep{BDF}. $\A(u)$ is introduced here for its usefulness in studying Goto numbers. Both have interesting connections to the Ap\'ery set of $S$ as shown in Lemma \ref{maxminchar} and Corollary \ref{Acor}.

\begin{thm}\label{main} Let $\Sgen$ be a semigroup and $0\ne u\in S$. Then the following integers are equal.
\begin{enumerate}
\item $g(u)$
\item $\min_{\alpha\in \A(u)}\big\{\,\max\{\ord(w) \st w\in\Ap(S,\alpha)\}\,\big\}$
\item $\min_{\alpha\in \A(u)} \big\{\,\max\{\ord(p+\alpha) \st p\in T\}\,\big\}$
\end{enumerate}
\end{thm}

\begin{proof} Let $R=k[[S]]$ be the ring corresponding to $S$. The integral closure of $x^uR$ is $x^uk[[x]] \cap R = \big(x^s\st s\in S$ and $s\ge u\big)R$.

(1) = (2) Let
$$N=\min_{\alpha\in \A(u)}\big\{\,\max\{\ord(w) \st w\in\Ap(S,\alpha)\,\big\}$$
We first show that $g(u) \le N$. For some $\alpha\in \A(u)$, $N = \max\{\ord(w) \st w\in\Ap(S,\alpha)\}$. We have $s-\alpha \in S$ for all $s\in S$ with $\ord(s)\ge N+1$. But $\alpha = u-b$ where $b\in S$ and $b<u$. So $s+b \in u+S$ for all $s\in S$ with $\ord(s)\ge N+1$. Therefore $x^b \in x^uR : m^{N+1}$ and it follows that $x^uR : m^{N+1}$ is not contained in the integral closure of $x^uR$. This proves that $g(u) \le N$.

Next we show that $N\le g(u)$. Let $b\in S$ with $b<u$ and choose $k$ such that $b+ka_1 < u \le b+(k+1)a_1$. Then $u-b-ka_1 \in \A(u)$. By the definition of $N$ there exists an $s\in S$ with $\ord(s) \ge N$ such that $s-u+b+ka_1 \not\in S$. Thus we have $s+b+ka_1 \not\in u+S$ which implies that $s+b \not\in u+S$. So $x^b \not\in x^uR : m^{N}$. Since $b$ was an arbitrary element of $S$ strictly less than $u$, we conclude that $x^uR : m^{N}$ is integral over $x^uR$. This shows that $N\le g(u)$.

(2) = (3) For a fixed $\alpha$ we show that $\max\{\ord(w) \st w\in\Ap(S,\alpha)\} = \max\{\ord(p+\alpha) \st p\in T\}$. Recall that by convention, if $p+\alpha\not\in \Ap(S;\alpha)$, then $\ord(p+\alpha)=-1$. Thus we need only consider the case when $p+\alpha \in \Ap(S;\alpha)$, and the inequality ``$\ge$" holds. For the reverse inequality ``$\le$", let $w\in \Ap(S;\alpha)$ have the largest order among the elements in $\Ap(S;\alpha)$. Then $w-\alpha \not\in S$, but $w -\alpha + s\in S$ for all $0\ne s\in S$. Therefore $w-\alpha = p \in T$ and $w=p+\alpha$.
\end{proof}

\begin{rmk}\label{mainr} A nice way to think of Theorem \ref{main}(3) is that $g(u) = \ord(p+\alpha)$ for some $\alpha \in \A(u)$ and $p \in T$ satisfying the following conditions.
\begin{enumerate}
\item For any $q \in T$, $\ord(p+\alpha) \ge \ord(q+\alpha)$
\item For any $\alpha' \in \A(u)$, there exists a $p'\in T$ such that $\ord(p+\alpha) \le \ord(p'+\alpha')$
\end{enumerate}
\end{rmk}

Notice that every Goto number of $S$ is determined by the order of $p+\alpha$ for some $p\in T$ and $\alpha \in \A(S)$. Moreover, for a fixed $\alpha$, we need only know the maximum order of $p+\alpha$ as $p$ ranges over the elements of $T$. If these values have been computed for $S$, then determining $g(u)$ is reduced to determining $\A(u)$. This motivates the following definition.

\begin{ddef}\label{gv gs} Let $\Sgen$ be a semigroup. For each $\alpha \in \A(S)$ let $\sigma(\alpha) = \max\{\ord(w) \st w\in \Ap(S;\alpha)\} = \max\{\ord(p+\alpha) \st p\in T\}$. Then $gv(S)=(\sigma(1),\dots,\sigma(e))$ will be called the {\em Goto vector of $S$}. Sometimes it is convenient to work with a set rather than a vector, so let $gs(S) = \{\sigma(1),\dots,\sigma(e)\}$ be the {\em Goto set of $S$}.
\end{ddef}

The next corollary follows immediately from Theorem \ref{main}.

\begin{cor}\label{cmain} Let $S$ be a semigroup, $gv(S)=(\sigma(1),\dots,\sigma(e))$ the Goto vector of $S$, and $0\ne u\in S$. Then $g(u) = \min\{\sigma(\alpha) \st \alpha\in \A(u)\}$.
\end{cor}

\begin{ex}\label{457} This example demonstrates how the Goto vector of a semigroup can be used to find the Goto numbers. Let $S=<4,5,7> = \{0,4,5,7,\rightarrow\}$ where the $\rightarrow$ indicates that all integers greater than 7 are in $S$. Then $T = \{3,6\}$ and $\A(S)=\{1,2,3,4\}$. We have $\sigma(1) = \max\{\ord(4),\ord(7)\} = 1$, $\sigma(2) = \max\{\ord(5),\ord(8)\} = 2$, $\sigma(3) = \max\{\ord(6),\ord(9)\} = 2$, and $\sigma(4) = \max\{\ord(7),\ord(10)\} = 2$. Notice that $6\not\in S$, so by convention $\ord(6)=-1$. Now the Goto vector of $S$ is $gv(S)=(1,2,2,2)$. For $0\ne u \in S$, if $1\in \A(u)$, then $g(u) = 1$ and otherwise $g(u) = 2$. Therefore
\begin{enumerate}
\item[] $g(4) =2$
\item[] $g(7) = 2$
\item[] $g(u) = 1$ for all other $0\ne u\in S$
\end{enumerate}
\end{ex}

The rest of this section is devoted to some immediate consequences of Theorem \ref{main}. Propositions \ref{gAp}, \ref{cons}, and \ref{consym} use the set $\A(u)$ to obtain results about $g(u)$. For a semigroup $S$ we set
\begin{eqnarray*}
\tau = \min\{g(s) \st 0\ne s\in S\}\\
\rho = \max\{g(s) \st 0\ne s\in S\}
\end{eqnarray*}

\begin{rmk}
Let $R=k[[S]]$ be the ring corresponding to $S$. We always have $\tau = \min\{g(Q) \st Q$ is a parameter ideal of $R\}$, but in general we only have $\rho \le \max\{g(Q) \st Q$ is a parameter ideal of $R\}$ \citep[Theorem 4.1, Example 4.4]{HS}.
\end{rmk}

\begin{prop}\label{gAp} Let $\Sgen$ be a semigroup. Then $g(a_1) = \max\{\ord(w) \st w\in \Ap(S)\}$.
\end{prop}

\begin{proof} $\A(a_1) = \{e\}$, and so the result follows from Theorem \ref{main}(2).
\end{proof}

\begin{prop}\label{cons} Let $\Sgen$ be a semigroup. Also let $u$, $u'$, and $u_i$ for all $i\in \I$ for some index set $\I$ be nonzero elements of $S$.
\begin{enumerate}
\item If $\A(u) \subset \A(u')$, then $g(u)\ge g(u')$
\item $g(u+u') \le \min\{g(u),g(u')\}$
\item $\rho = \max\{g(a_i) \st 1\le i\le \nu\}< \infty$
\item\label{goto A min} $\tau = g(u)$ for all $u\ge f+a_1+1$
\item $\tau = \min\{\sigma(1),\sigma(2),\dots,\sigma(e)\}$
\item\label{goto A U} If\, $\bigcup_{i\in \I} \A(u_i) = \A(S)$, then $\tau = \min\{g(u_i)\st i\in \I\}$
\item\label{min a1} $\tau = \min\{g(a_1),g(f+a_1)\}$
\end{enumerate}
\end{prop}

\begin{proof} (1) If $\A(u) \subset \A(u')$, then $g(u)  = \min\{\sigma(\alpha) \st \alpha\in \A(u)\} \ge \min\{\sigma(\alpha) \st \alpha\in \A(u')\} = g(u')$.

(2) This follows from (1) since $\A(u) \subset \A(u+u')$ and $\A(u') \subset \A(u+u')$.

(3) Note that if $0\ne u\in S$, then $u = \sum_{i=1}^{\nu} c_ia_i$ where $c_j>0$ for at least one $j$. But then by (2), $g(u)\le g(a_j)$. This shows that $\rho = \max\{g(a_i) \st 1\le i\le \nu\}$. Furthermore $g(a_i) <\infty$ for all $i$ and the set of minimal generators is finite, so $\max\{g(a_i) \st 1\le i\le \nu\}< \infty$.

(4) If $u \ge f+a_1+1$, then $\A(s) \subset \A(u) = \A(S)$ for all $0\ne s\in S$. So by (1), $g(u) \le g(s)$ for all $0\ne s\in S$.

(5) $\tau = g(f+a_1+1) = \min\{\sigma(\alpha) \st \alpha\in \A(S)\}= \min\{\sigma(1),\sigma(2),\dots,\sigma(e)\}$.

(6) There exists a $\beta$ such that $\sigma(\beta) = \min\{\sigma(\alpha) \st \alpha \in \A(S)\}$ and by hypothesis $\beta \in \A(u_j)$ for some $u_j$. Thus $\tau = g(u_j) = \min\{g(u_i)\st i\in \I\}$.

(7) This follows from (6) since $\A(a_1) = \{e\}$ and $\A(f+a_1) = \{1,2,\dots,e-1\}$.
\end{proof}

Statement (2) of Proposition \ref{cons} is a special case of \citep[Corollary 1.7]{HS}, and (3) and (4) were also shown by \citet[Proposition 4.3, Theorem 4.1]{HS}.

Next we consider symmetric semigroups.

\begin{ddef}\label{symmetric} A semigroup $S$ with Frobenius number $f$ is called {\em symmetric} if whenever $x+y = f$ for $x,y\in \Z$, then exactly one of $x$ and $y$ belongs to $S$. Equivalently, $S$ is symmetric if exactly half of the elements in $\{0,1,\dots,f\}$ are in $S$.
\end{ddef}

Several characterizations of the symmetric property are given in Proposition \ref{psym}. For Proposition \ref{symgv}, which follows immediately from Corollary \ref{cmain}, it suffices to know that $S$ is symmetric if and only if $T=\{f\}$.

\begin{prop}\label{symgv} Let $S$ be a symmetric semigroup with Frobenius number $f$ and $0\ne u\in S$. Then $gv(S) = (\ord(f+1),\ord(f+2),\dots,\ord(f+a_1))$ and $g(u) = \min\{\ord(f+\alpha)\st \alpha\in \mathcal A(u)\}$.
\end{prop}

The next proposition is similar to Proposition \ref{cons} with the additional assumption that $S$ is symmetric.

\begin{prop}\label{consym} Let $\Sgen$ be a symmetric semigroup with Frobenius number $f$ and $0\ne u\in S$.
\begin{enumerate}
\item $g(a_1) = \ord(f+a_1)$
\item $g(a_2) = \ord(f+a_2-\big\lfloor\frac{a_2}{a_1}\big\rfloor a_1)$
\item\label{prop3 2} $\tau = \min\{\ord(f+1),\ord(f+2),\dots,\ord(f+a_1)\}$
\item If $\alpha \in \A(u)$, then $\tau = \min\{g(u),g(f+\alpha)\}$
\item $\tau = g(u)$ for all $u> f$ if the Goto vector does not have a unique least entry
\item Exactly one element $u$ of $S$ with $u>f$ does not obtain the minimal Goto number if the Goto vector has a unique least entry
\end{enumerate}
\end{prop}

\begin{proof} (1) and (2) We only need to notice that $\A(a_1) = \{e\}$ and $\A(a_2) = \{a_2-\big\lfloor\frac{a_2}{a_1}\big\rfloor a_1\}$.

(3) This follows from Proposition \ref{cons}(5).

(4) Using Proposition \ref{cons}(6), it suffices to show that $\A(f+\alpha) = \{\A(S)\setminus \alpha\}$. Clearly $\alpha \not\in \A(f+\alpha)$. If $\alpha\ne \beta \in \A(S)$, then $\beta - \alpha < a_1$ and $\beta-\alpha \ne 0$. So $\beta - \alpha \not\in S$ and since $(f+\alpha-\beta) + (\beta-\alpha) = f$, we have $f+\alpha-\beta\in S$. Therefore $\beta \in \A(f+\alpha)$.

(5) We only need to check for $g(s)$ when $f< s < f+a_1+1$. In this case, $s = f+\alpha$ for some $\alpha \in \A(S)$. As in (4), we have that $\A(s) = \A(S)\setminus\{\alpha\}$. Since the Goto vector has more than one least entry, there exists $\beta \in \A(s)$ such that $\sigma(\beta) = \min\{\sigma(1),\dots,\sigma(a_1)\}$.

(6) Let $\sigma(\beta)$ be the unique least entry of the Goto vector. Then $g(f+\beta) = \min\{\sigma(\alpha) \st \alpha \in \A(S)\setminus\{\beta\}\} > \tau$. For any other element $s\in S$ with $s> f$, $\beta \in \A(s)$. Thus $g(s) = \tau$.
\end{proof}

Theorem \ref{main} (or Remark \ref{mainr}) enables us to establish bounds for the Goto numbers by bounding the order of an element of $S$. For $s\in S$, $\ord(s)=\max\{\sum_{i=1}^{\nu} c_i \st s=\sum_{i=1}^{\nu} c_ia_i$ where $c_i\ge 0\}$. A sum $\sum_{i=1}^{\nu} c_ia_i$ is called a {\em representation} of the element $s\in S$ if $s=\sum_{i=1}^{\nu} c_ia_i$. If in addition $\sum_{i=1}^{\nu} c_i = \ord(s)$, we say that it is a {\em maximal representation}. The next lemma gives some bounds on $\ord(s)$.

\begin{lem}\label{about ord} Let $\Sgen$ be a semigroup and $s\in S$.
\begin{enumerate}
\item If $j>0$ is an integer such that there exists a maximal representation $ s=\sum^j_{i=1} c_ia_i$, then $\big\lceil\frac{s}{a_j}\big\rceil \le \ord(s)$.
\item If $j'>0$ is an integer such that there exists a maximal representation $s=\sum^{\nu}_{i=j'} c_ia_i$ then $\ord(s) \le \big\lfloor\frac{u}{a_{j'}}\big\rfloor$.
\end{enumerate}
\end{lem}

\begin{proof} For (1), suppose that $\ord(s) \le \big\lceil\frac{s}{a_j}\big\rceil-1$. Then $s \le (\big\lceil\frac{s}{a_j}\big\rceil-1)a_j \Rightarrow s \le \big\lceil\frac{s}{a_j}\big\rceil a_j-a_j$ which is a contradiction. Likewise for (2), suppose that $\ord(s) \ge \big\lfloor\frac{s}{a_{j'}}\big\rfloor+1$. Then $s \ge (\big\lfloor\frac{s}{a_{j'}}\big\rfloor+1)a_{j'} \Rightarrow s \ge \big\lfloor\frac{s}{a_{j'}}\big\rfloor a_1+a_1$ which is a contradiction.
\end{proof}

\begin{prop}\label{bounds} Let $\Sgen$ be a semigroup with Frobenius number $f$ and $0\ne u\in S$. Also let $\wt u$ be the largest element in $S$ that is strictly smaller than $u$.
\begin{enumerate}
\item[{Upper}] {\em Bounds}
\item $g(u) \le \big\lfloor\frac{f+u-\wt u}{a_1}\big\rfloor$
\item $g(a_1) \le \big\lfloor\frac{f+a_1}{a_2}\big\rfloor$
\item $\rho \le \big\lceil\frac{f}{a_1}\big\rceil$
\item[{Lower}] {\em Bounds}
\item $g(u) \ge \big\lceil\frac{f+u-\wt u}{a_\nu}\big\rceil$
\item $g(a_\nu) \ge \big\lceil\frac{f+a_\nu-\wt{a_\nu}}{a_{\nu-1}}\big\rceil$
\item $\tau \ge \sceil[a_\nu]{f}$
\end{enumerate}
\end{prop}

\begin{proof} Let $g(u) = \ord(p+\alpha)$ as in Remark \ref{mainr}. Notice that $f$ is the largest element of $T$ and $u-\wt u$ is the least element of $\A(u)$. For (1), we have

$$g(u) = \ord(p+\alpha) \stackrel{\ref{mainr}(2)}{\le} \ord(p'+(u-\wt u)) \stackrel{\ref{about ord}}{\le}  \floor{p'+u-\wt u} \stackrel{}{\le} \floor{f+u-\wt u}$$

\noindent For (2), $\A(a_1) = \{e\}$ so

$$g(a_1) = \ord(p+a_1) \stackrel{\ref{about ord}}{\le} \floor[a_2]{p+a_1} \le \floor[a_2]{f+a_1}$$
\noindent Since $1\le u-\wt u\le a_1$, $\sfloor{f+u-\wt u} \le \sceil{f}$ and (3) follows from (1).

Similarly for (4), we have

$$g(u) = \ord(p+\alpha) \stackrel{\ref{mainr}(1)}{\ge} \ord(f+\alpha) \stackrel{\ref{about ord}}{\ge} \ceil[a_\nu]{f+\alpha} \stackrel{}{\ge} \ceil[a_\nu]{f+u-\wt u}$$

\noindent For (5), notice that if $\alpha \in \A(a_\nu)$, then $f+\alpha \in \Ap(S,a_\nu)$. Suppose not, then $f+\alpha - a_\nu \in S$. However, $a_\nu -\alpha \in S$ and so $f = (f+\alpha - a_\nu) + (a_\nu -\alpha) \in S$ which is a contradiction. Now

$$g(a_\nu) = \ord(p+\alpha) \stackrel{\ref{mainr}(1)}{\ge} \ord(f+\alpha) \stackrel{\ref{about ord}}{\ge} \ceil[a_{\nu-1}]{f+\alpha} \stackrel{}{\ge} \ceil[a_{\nu-1}]{f+a_{\nu}-\wt{a_{\nu}}}$$

\noindent Lastly for (6), $1\le u-\wt u\le a_\nu$. So $\big\lceil\frac{f+u-\wt u}{a_\nu}\big\rceil \ge \sfloor[a_\nu]{f+a_\nu} = \sfloor[a_\nu]{f}+1 = \sceil[a_\nu]{f}$ and (6) follows from (4).
\end{proof}

Upper Bounds of this kind were given for the minimal generators of $S$ by \citet[Proposition 5.1, Proposition 5.3]{HS}.

\begin{rmk}\label{3.2} In fact we have $\sceil[a_\nu]{f} \le g(Q) \le \sceil{f}$ for {\em every} parameter ideal $Q$ of $R=k[[S]]$, not just the monomial parameter ideals. This follows from \citep[Theorem 4.1, Theorem 4.7]{HS} combined with Proposition \ref{bounds}(6). This upper bound is not always sharp \citep[Remark 4.9]{HS}. Likewise the example $S=<5,8,12>$ with Frobenius number 19 shows that this lower bound is also not always sharp since $\sceil[a_3]{f}=2$ and $\tau=3$.
\end{rmk}

\begin{rmk}\label{3.3} If $R = k[[x]]$, the Goto number of every parameter ideal is 0. This is the case for any regular local ring of dimension 1. If $R \ne k[[x]]$, then $a_\nu$ and $f$ are positive and hence $\sceil[a_\nu]{f} \ge 1$. So the Goto number of any parameter ideal in $R$ is always greater than or equal to 1. \citet[Theorem 2.2]{CP} showed that this holds for all Cohen-Macaulay local rings that are not regular. If $R$ is Gorenstein with multiplicity greater than 2, then $a_\nu < f$. So $\sceil[a_\nu]{f} \ge 2$ and the Goto number of any parameter ideal in $R$ is always greater than or equal to 2. This was also shown by \citet[Proposition 2.5]{GKM}, and more generally for local Gorenstein rings with positive dimension and multiplicity greater than 2 by \citet[Theorem 1.1]{GMT}.
\end{rmk}


\section{$M$-pure Semigroups and the Gorenstein Property}\label{prelim}

Given a partial ordering on $S$, we can consider the minimal and maximal elements of $\Ap(S)\setminus\{0\}$. For the partial ordering $\pre$ defined in the beginning of Section 2 (recall that $u \pre u'$ if $u+s = u'$ for some $s\in S$), these are denoted by $\minap(S)$ and $\maxap(S)$. In this section a different partial ordering is also used, namely $\pre_M$ where $u \pre_M u'$ if $u+s = u'$ and $\ord(u) + \ord(s) = \ord(u')$ for some $s\in S$. The corresponding sets are denoted by $\minap_M(S)$ and $\maxap_M(S)$.

A remark should be made about notation. Here $M = S\setminus \{0\}$ denotes the maximal ideal of $S$, see \citep{BDF} for more information about ideals of semigroups. $M$ is used in the notation because the $\ord(u) = k$ such that $u \in kM\setminus (k+1)M$, and hence this order function is induced by the $M$-adic filtration $S\supset M \supset 2M\supset \dots$. In general a partial ordering can be defined using an order function induced by any filtration on $S$.

\begin{ddef}\label{purity}\
\begin{enumerate}
\item $S$ is called {\em pure} if every element in $\maxap(S)$ has the same order.
\item $S$ is called {\em $M$-pure} if every element in $\maxap_M(S)$ has the same order.
\end{enumerate}
\end{ddef}

Proposition \ref{mpiffp=} states the connection between purity and $M$-purity. First we need two lemmas. Statement (2) of Lemma \ref{maxminchar} is also given by \citet[Theorem 7]{FGH}.

\begin{lem}\label{maxminchar} Let $\Sgen$ be a semigroup. Then
\begin{enumerate}
\item $w\in \minap(S)$ if and only if $w = a_i$ for some $2\le i\le \nu$.
\item $w\in \maxap(S)$ if and only if $w-a_1\in T$.
\end{enumerate}
\end{lem}

\begin{proof} (1) Let $w\in \minap(S)$ and suppose that $w \ne a_i$ for some $2\le i\le \nu$. Since $a_1 \not\in \Ap(S)$ and $0\not\in \minap(S)$, we also have that $w$ is nonzero and not equal to $a_1$. Thus there exists $0\ne u \in S$ with $u<w$ such that $u\pre w$. However, $u$ must be in $\Ap(S)$ contradicting that $w\in \minap(S)$. Conversely suppose that $w \not\in \minap(S)$. Then there exists $0\ne u \in \Ap(S)$ with $u<w$ such that $u\pre w$. Therefore $w$ cannot be a minimal generator.

(2) If $w\in \maxap(S)$, then $w-a_1 \not\in S$. However $w+u \not\in \Ap(S)$ for all $0\ne u\in S$ so that $w-a_1+u \in S$. Thus $w-a_1\in T$. The converse is clear.
\end{proof}

\begin{lem}\label{premaxrep} Let $\Sgen$ be a semigroup and $u' = \sum_{i=1}^{\nu} c_ia_i$ be a maximal representation of $u'\in S$. If $u = \sum_{i=1}^{\nu} d_ia_i$ with $0\le d_i\le c_i$ for all $1\le i\le \nu$, then $u \pre_M u'$.
\end{lem}

\begin{proof} Let $s = \sum_{i=1}^{\nu}(c_i-d_i)a_i$, then $s\in S$, $u+s = u'$. Now we have $\ord(u) + \ord(s) \ge  \sum_{i=1}^{\nu} c_i + \sum_{i=1}^{\nu} (c_i-d_i) = \ord(u')$. However we always have $\ord(u) + \ord(s)\le \ord(u')$.
\end{proof}

\begin{prop}\label{mpiffp=} Let $\Sgen$ be a semigroup. Then
\begin{enumerate}
\item $\minap(S) = \minap_M(S)$
\item $\maxap(S) \subset \maxap_M(S)$
\item $S$ is $M$-pure if and only if $S$ is pure and $\maxap(S) = \maxap_M(S)$
\end{enumerate}
\end{prop}

\begin{proof} For (1), if $w \not\in \minap_M(S)$, then $w_i+w_j = w$ for some $w_i,w_j \in \Ap(S)\setminus\{0\}$. Hence $w \not\in \minap(S)$. Now let $w \in \minap_M(S)$ and suppose that $\ord(w) \ge 2$. Then there is a maximal representation $w = \sum_{i=1}^{\nu} c_ia_i$ such that $\sum_{i=1}^{\nu} c_i \ge 2$. So there exists $w_i \in \Ap(S)$ such that $w_i < w$ and $w_i = \sum d_ia_i$ with $0\le d_i\le c_i$ for all $1\le i\le \nu$. By Lemma \ref{premaxrep}, $w_i \pre_M w$ which is a contradiction. Therefore $\ord(w) = 1$ and $w$ is a minimal generator of $S$. By Lemma \ref{maxminchar}, $w\in \minap(S)$.

For (2), if $w \not\in \maxap_M(S)$, then $w_i+w = w_j$ for some $w_i,w_j \in \Ap(S)\setminus\{0\}$. Hence $w \not\in \maxap(S)$.

To see (3), first assume $S$ is $M$-pure. Since $\maxap(S) \subset \maxap_M(S)$, $S$ is pure. Moreover $\maxap(S)$ contains every element of $\Ap(S)$ with the largest order among the elements in $\Ap(S)$. Again since $\maxap(S) \subset \maxap_M(S)$, every element of $\maxap_M(S)$ has this largest order and is thus in $\maxap(S)$. The reverse implication is trivial.
\end{proof}

\begin{ex}\label{569} Let $S=<5,6,9>$. Then $\Ap(S) = \{0,6,9,12,18\}$, $\minap(S)=\{6,9\}$, $\maxap(S)=\{18\}$, and $\maxap_M(S)=\{9,18\}$. So $S$ is pure, but not $M$-pure.
\end{ex}

Notice, as in the example, that every symmetric semigroup is pure, but not necessarily $M$-pure. The next proposition is a collection of well-known equivalent formulations of symmetry, and proofs are provided for completeness. In Proposition \ref{mpsym} an analogous result is given for $M$-pure symmetry.

\begin{prop}\label{psym} Let $\Sgen$ be a semigroup with Frobenius number $f$ and $\Ap(S) =\{w_0,w_1,\dots,w_{e-1}\}$. Then the following are equivalent.
\begin{enumerate}
\item $S$ is (pure) symmetric
\item $\maxap(S) = \{f+a_1\}$
\item $\#\maxap(S) = 1$
\item $w_i + w_j = w_{e-1}$ whenever $i+j = e-1$
\item $w_i \pre w_{e-1}$ for $0\le i \le e-1$
\end{enumerate}
\end{prop}

\begin{proof} (1) $\Rightarrow$ (2) Let $w \in \maxap(S)$. Then $(w-a_1) + y =f$ for some $y\in S$ and by Lemma \ref{maxminchar}, $w-a_1 \in T$. Thus $y=0$ and $w = f+a_1$.

(2) $\Rightarrow$ (3) Clear.

(3) $\Rightarrow$ (4) $\maxap(S) =\{w_{e-1}\}$, and so $w_{e-1}-w\in \Ap(S)$ if and only if $w\in \Ap(S)$. Therefore $\Ap(S) = \{w_{e-1} - w_{e-1}, w_{e-1} - w_{e-2},\dots,w_{e-1} - w_{0}\}$. Now we can see that $w_i = w_{e-1} - w_{e-1-i}$.

(4) $\Rightarrow$ (5) Clear.

(5) $\Rightarrow$ (1) Let $x+y = f$ for $x,y\in \Z$ and suppose that $x \not\in S$. Choose $k>0$ such that $x+ka_1 \in \Ap(S)$. Then $(x+ka_1) +(y-(k-1)a_1) = f+a_1 = w_{e-1}$ and $(y-(k-1)a_1)\in S$. Since $k-1\ge0$, $y \in S$.
\end{proof}

\begin{prop}\label{mpsym} Let $\Sgen$ be a semigroup with Frobenius number $f$ and $\Ap(S) =\{w_0,w_1,\dots,w_{e-1}\}$. Then the following are equivalent.
\begin{enumerate}
\item $S$ is $M$-pure symmetric
\item $\maxap_M(S) = \{f+a_1\}$
\item $\#\maxap_M(S) = 1$
\item $w_i + w_j = w_{e-1}$ and $\ord(w_i) + \ord(w_j) = \ord(w_{e-1})$ whenever $i+j = e-1$
\item $w_i \pre_M w_{e-1}$ for $0\le i \le e-1$.
\end{enumerate}
\end{prop}

\begin{proof} (1) $\Rightarrow$ (2) $\maxap_M(S) = \maxap(S)$ since $S$ is $M$-pure, and $\maxap(S)=\{f+a_1\}$ since $S$ is symmetric.

(2) $\Rightarrow$ (3) Clear.

(3) $\Rightarrow$ (4) Since $\emptyset \ne \maxap(S) \subset \maxap_M(S)$, $\#\maxap(S) = 1$. By Proposition \ref{psym} $w_i + w_j = w_{e-1}$ whenever $i+j = e-1$ and $\maxap_M(S) =\maxap(S) =\{f+a_1\}$. But $f+a_1 = w_{e-1}$, so $\ord(w_i) + \ord(w_j) = \ord(w_{e-1})$.

(4) $\Rightarrow$ (5) Clear.

(5) $\Rightarrow$ (1) Since $w_i \pre_M w_{e-1}$ implies that $w_i \pre w_{e-1}$, by Proposition \ref{psym} $S$ is symmetric. Moreover, clearly $\maxap_M(S) = \{w_{e-1}\}$, and so $S$ is $M$-pure.
\end{proof}

Let $S$ be an $M$-pure symmetric semigroup, $g=g(a_1)$, and $\beta_i = \#\{w\in \Ap(S)\st \ord(w)=i\}$. Notice we have $\{w\in \Ap(S) \st \ord(w)=i\} = \{w_{e-1} - w\st w\in \Ap(S)$ and $\ord(w) = g-i\}$. Thus $\beta_i = \beta_{g-i}$ for $0\le i\le \floor[2]{g+1}$. This is in fact a sufficient condition for $M$-purity when $S$ is symmetric as shown in the next proposition. Proposition \ref{beta stuff} is a key result for showing that $S$ is  $M$-pure symmetric if and only if the associated graded ring $\gr_{\bar m}{\bar R}$, where $\bar R = R/x^{a_1}R$ and $\bar m = m/x^{a_1}R$, is Gorenstein.

\begin{prop}\label{beta stuff} Let $S$ be a symmetric semigroup. Then we have the following.
\begin{enumerate}
\item $\sum_{j=0}^i \beta_j \ge \sum_{j=0}^i \beta_{g-j}$ for $0\le i\le g$
\item The following are equivalent
\begin{enumerate}
\item $\sum_{j=0}^i \beta_j = \sum_{j=0}^i \beta_{g-j}$ for $0\le i\le g$
\item $\beta_i = \beta_{g-i}$ for $0\le i\le \floor[2]{g+1}$
\item $S$ is $M$-pure
\end{enumerate}
\end{enumerate}
\end{prop}

\begin{proof} Let
\begin{eqnarray*}
A_i &=& \{w\in \Ap(S) \st \ord(w) \ge g-i\}\\
B_i &=& \{w\in \Ap(S) \st w+w'\not\in\Ap(S) \text{ if } w'\in\Ap(S) \text{ and } \ord(w')>i\}\\
C_i &=& \{w_{e-1} - w' \st w'\in\Ap(S) \text{ and } \ord(w') \le i\}
\end{eqnarray*}

By Proposition \ref{psym}, $C_i$ is a subset of $\Ap(S)$. We always have $A_i \subset B_i = C_i$. Indeed let $w \in A_i$ and $w' \in \Ap(S)$ with $\ord(w')>i$. Then $\ord(w+w') \ge \ord(w) + \ord(w') > g-i + i = g$. Thus $w + w' \not\in \Ap(S)$ and $w \in B_i$. Now let $w\in B_i$. By Proposition \ref{psym} there exists a $w'\in \Ap(S)$ such that $w = w_{e-1} - w'$, and since $w\in B_i$ we have $\ord(w') \le i$. Therefore $w \in C_i$. Lastly let $w\in C_i$. That is $w+w' = w_{e-1}$ where $w'\in \Ap(S)$ and $\ord(w') \le i$. Suppose there exists $w''\in \Ap(S)$ with $\ord(w'')> i$ such that $w+w''\in \Ap(S)$. Then by Proposition \ref{psym}, $w''\pre w'$, which is a contradiction. So $w\in B_i$.

Now $\sum_{j=0}^i \beta_{g-j} = \#A_i \le \#B_i = \#C_i = \#\{w\in \Ap(S)\st \ord(w) \le i\} = \sum_{j=0}^i \beta_j$. This proves (1). The implications $(2a) \iff (2b) \Leftarrow (2c)$ are clear. It remains to show that $(2a) \Rightarrow (2c)$.

Assume that we have $\sum_{j=0}^i \beta_j = \sum_{j=0}^i \beta_{g-j}$ for $0\le i\le g$. Then $A_i = B_i$ for all $i$. Let $w_i + w_j = w_{e-1}$ and note that if $w\in \Ap(S)$ with $\ord(w) > \ord(w_j)$, then $w_i + w\not\in \Ap(S)$. Thus $w_i \in B_{\ord(w_j)} = A_{\ord(w_j)}$ and $\ord(w_i) \ge g-\ord(w_j)$. Since we always have $\ord(w_i) \le g-\ord(w_j)$, it follows that $\ord(w_i) + \ord(w_j) = \ord(w_{e-1})$ and $S$ is $M$-pure by Proposition \ref{mpsym}.
\end{proof}

\begin{rmk}
It is necessary to assume that $S$ is symmetric. If $S$ is not symmetric, then $M$-purity implies that $\beta_0 < \beta_g$. On the other hand consider the semigroup $S=<4,5,11>$. Then $\Ap(S) = \{0,5,10,11\}$ and $\maxap_M(S) = \{10,11\}$. Since $\ord(10)=2$ and $\ord(11)=1$, $S$ is not $M$-pure. However, $\beta_0=1$, $\beta_1=2$, and $\beta_2=1$.
\end{rmk}

Just as $S$ being symmetric is equivalent to $R$ being Gorenstein \citep{K}, again an analogous statement holds for $M$-pure symmetric. We consider this now. More generally, it is determined how the $M$-purity of $S$ is reflected in the corresponding numerical semigroup ring $R=k[[S]]$. First some results about Goto numbers and indices of nilpotency are needed.

As noted by \citet[Remark 3.3]{HS} sometimes the Goto number of a parameter ideal in a Noetherian local ring is equal to the index of nilpotency of the maximal ideal with respect to that parameter ideal. The next lemma states that this is the case precisely when $Q$ is a reduction of the maximal ideal.

\begin{lem}\label{g=s} Let $(A,m)$ be a Noetherian local ring with parameter ideal $Q$. Also let $s=s_Q(m) = \min\{k \st m^{k+1} \subset Q\}$ be the index of nilpotency of $m$ with respect to $Q$ and $g=g(Q)$ be the Goto number of $Q$. Then $g\le s$ with equality holding if and only if $Q$ is a reduction of $m$.
\end{lem}

\begin{proof} Since $Q:m^{s+1} = A$, by Definition \ref{goto def}, $g\le s$. If $g = s$, then $Q \subset m \subset Q:m^s = Q:m^g \subset \bar Q$ where $\bar Q$ is the integral closure of $Q$. Thus $Q$ is a reduction of $m$. Conversely if $Q$ is a reduction of $m$, then $Q:m^s = m\subset \bar Q$ and $g\ge s$, hence $g=s$.
\end{proof}

\begin{cor}\label{gq=s} Let $\Sgen$ be a semigroup with corresponding ring $R=k[[S]]$, maximal ideal $m$, and $Q=x^{a_1}R$. Then $Q$ is a reduction of $m$ and $g(a_1) = s_Q(m)$.
\end{cor}

\begin{prop}\label{mpure} Let $\Sgen$ be a semigroup and $g=g(a_1)$. Also let $R=k[[S]]$ be the corresponding ring with maximal ideal $m$. Then the following are equivalent.
\begin{enumerate}
\item $S$ is $M$-pure.
\item $(x^{a_1}):m^n = (x^{a_1}) + m^{g-n+1}$ for $1\le n\le g$.
\end{enumerate}
\end{prop}

\begin{proof} (1) $\Rightarrow$ (2) $m^{g-n+1}m^n = m^{g+1} \stackrel{\ref{gq=s}}{\subset} (x^{a_1})$. So $(x^{a_1}) + m^{g-n+1} \subset (x^{a_1}):m^n$. Now it suffices to show that if $x^{w_i} \in (x^{a_1}) :m^n$, then $x^{w_i} \in m^{g-n+1}$ where $w_i \in \Ap(S)$. By hypothesis there exists $w_j \in \Ap(S)$ such that $w_i+w_j \in \Ap(S)$ and $\ord(w_i) + \ord(w_j) = g$. Since $x^{w_i} \in (x^{a_1}) :m^n$, $\ord(w_j)<n$. Therefore $\ord(w_i) = g-\ord(w_j) > g-n$, and so $x^{w_i} \in m^{g-n+1}$.

For (2) $\Rightarrow$ (1) if $w_i \in \maxap(S)$, then $x^{w_i} \in (x^{a_1}):m = (x^{a_1}) + m^{g}$. So $\ord(w_i) = g = \max\{\ord(w)\st w\in \Ap(S)\}$, and $S$ is pure. Now suppose that $w_i \in \maxap_M(S) \setminus \maxap(S)$. Then the set $U = \{w\in \Ap(S)\setminus\{0\} \st w_i+w \in \Ap(S)\}$ is nonempty. Let $w_j \in U$ be an element with the largest order among the elements of $U$. By our choice of $w_j$, $x^{w_i} \not\in (x^{a_1}):m^{\ord(w_j)} = (x^{a_1}) + m^{g-\ord(w_j)+1}$, but $x^{w_i} \in (x^{a_1}):m^{\ord(w_j)+1} = (x^{a_1}) + m^{g-\ord(w_j)}$. Thus $\ord(w_i) = g-\ord(w_j)$, and we have $g = \ord(w_i) +\ord(w_j) \le \ord(w_i+w_j)\le g$. This implies that $w_i \preceq_M w_i+w_j$ which is a contradiction. Therefore by Proposition \ref{mpiffp=}, the implication follows.
\end{proof}

The following results of \citet{G} are needed in Theorem \ref{equivs}.

\begin{lem}\label{garcia}\citep[Theorem 7, Remark 8]{G} Let $\Sgen$ be a semigroup and $R=k[[S]]$ with maximal ideal $m$ its corresponding ring. Then the following are equivalent.
\begin{enumerate}
\item $\gr_m(R)$ is Cohen-Macaulay.
\item The image of $x^{a_1}$ in $\gr_m(R)$ is a nonzerodivisor.
\item $\ord(u+a_1) = \ord(u) + 1$ for all $u\in S$.
\item $\ord(w+ ka_1) = \ord(w) + k$ for all $w \in \Ap(S)$ and $k\ge0$.
\end{enumerate}
\end{lem}

\begin{thm}\label{equivs} Let $\Sgen$ be a semigroup, $R=k[[S]]$ the corresponding ring with maximal ideal $m$, and $r = r(m)$ the reduction number of $m$ (with respect to $x^{a_1}R$). Also let $\gr_m(R)$ be the associated graded ring of $R$ with respect to the maximal ideal $m$, and $\gr_{\bar m}{\bar R}$ the associated graded ring of $\bar R = R/x^{a_1}R$ with respect to $\bar m = m/x^{a_1}R$.
\begin{enumerate}
\item $S$ is $M$-pure symmetric if and only if $\gr_{\bar m}{\bar R}$ is Gorenstein.
\item If $\gr_m(R)$ is Cohen-Macaulay, then $g(a_1) = r$. The converse holds if $S$ is $M$-pure.
\item $S$ is $M$-pure symmetric and $g(a_1) = r$ if and only if $\gr_m(R)$ is Gorenstein.
\end{enumerate}
\end{thm}

\begin{proof} (1) Let $\gr_{\bar m}(\bar R) = \bigoplus_{i=0}^g G_i$ where $G_i = \bar m^i/\bar m^{i+1}$ and $g=g(a_1)$. That $G_i = 0$ for $i> g$ follows from Corollary \ref{gq=s}. Also let $\beta_i = \#\{w\in \Ap(S)\st \ord(w) = i\}$. It is not difficult to see that $\beta_i = \dim_k G_i$. Now $S$ is $M$-pure symmetric $\iff$ $R$ is Gorenstein and $\beta_i = \beta_{g-i}$ for $0\le i\le \floor[2]{g+1}$ $\iff$ $\bar R$ is Gorenstein and $\dim_k G_i = \dim_k G_{g-i}$ for $0\le i\le \floor[2]{g+1}$. This last statement is equivalent to the Gorensteiness of $\gr_{\bar m}(\bar R)$ \citep[Theorem 3.1]{HKU} or \citep[Theorem 1.5]{O}.

(2) First note that $r$ is the least integer $k$ such that both of the following conditions hold:
\begin{enumerate}
\item[(i)] $\ord(u+a_1) = \ord(u) + 1$ whenever $\ord(u) \ge k$.
\item[(ii)] $u\not\in \Ap(S)$ whenever $\ord(u) > k$.
\end{enumerate}
Indeed if $\ord(u+a_1) =r+n+1$ for some $n\ge 0$, then $x^{a_1}x^u = x^{u+a_1} \in m^{r+n+1} = x^{a_1}m^{r+n}$. Thus $x^u \in m^{r+n}$, and $\ord(u) = r+n = \ord(u+a_1) - 1$. Moreover if $\ord(u) > r$, then $x^u \in m^{r+1} = x^{a_1}m^r$. This implies that $u = a_1 + s$ for some $s\in S$, and so $u\not\in \Ap(S)$. This establishes that $r$ satisfies these conditions. If $r'$ satisfies (i) and (ii), then it is clear that $m^{r'+1} = x^{a_1}m^{r'}$. Thus $r\le r'$.

If $\gr_m(R)$ is Cohen-Macaulay, by Lemma \ref{garcia}, condition (i) is satisfied for all integers $k\ge 0$. Thus $r$ is least integer $k$ satisfying condition (ii), which is $g$ by Proposition \ref{gAp}. To see that the converse holds when $S$ is $M$-pure let $w_i +ka_1 \in S$ for some $k\ge 0$. Then there exists an element $w_j \in \Ap(S)$ such that $w_i+w_j \in \maxap_M(S)$. Notice that $r = \ord(w_i+w_j)$. We now have
\begin{eqnarray*}
\ord(w_i) +k + \ord(w_j) &\le& \ord(w_i +ka_1) + \ord(w_j)\\
&\le& \ord(w_i +w_j+ka_1)\\
&=& \ord(w_i +w_j)+k\\
&=& \ord(w_i) +\ord(w_j)+k\\
\end{eqnarray*}
Thus $\ord(w_i) +k = \ord(w_i +ka_1)$ and $\gr_m(R)$ is Cohen-Macaulay.

(3) Using (1) and (2), $S$ is $M$-pure symmetric and $g(a_1) = r \iff \gr_{\bar m}(\bar R)$ is Gorenstein and $\gr_m(R)$ is Cohen-Macaulay $\iff$ $\gr_m(R)$ is Gorenstein.
\end{proof}

\begin{exs}\label{show equivs} With notation as in Theorem \ref{equivs},
\begin{enumerate}
\item Let $S=<10,17,35>$. Then $\Ap(S) = \{0, 17, 34, 35, 51, 52, 68, 69, 86, 103\}$ and their orders are $\{0,1,2,1,3,2,4,3,4,5\}$. So $S$ is $M$-pure symmetric. Moreover, $r=5$, so $g(a_1) = r$. Therefore $\gr_m(R)$ is Gorenstein.
\item Let $S=<6,7,15>$. Then $\Ap(S) = \{0,7,14,15,22,29\}$ and their orders are $\{0,1,2,1,2,3\}$. So $S$ is $M$-pure symmetric. However, $r = 5$, so $g(a_1) < r$. Thus $\gr_{\bar m}(\bar R)$ is Gorenstein, but $\gr_m(R)$ is not.
\item Let $S=<5,8,12>$. Then $\Ap(S) = \{0,8,12,16,24\}$ and their orders are $\{0,1,1,2,3\}$. So $S$ is symmetric, but not $M$-pure. Hence $R$ is Gorenstein, but $\gr_{\bar m}(\bar R)$ is not.
\end{enumerate}
\end{exs}

\begin{ex}\label{78919} With notation as in Theorem \ref{equivs}, the example $S=<7,8,9,19>$ shows that $g(a_1) = r$ is not in general equivalent to the Cohen-Macaulayness of $\gr_m(R)$. We have $g(a_1) = r = 3$, but $\gr_m(R)$ is not Cohen-Macaulay by Lemma \ref{garcia} since $\ord(7+19) = \ord(8+(2)9) = 3$. \citet[Theorem 4.12]{S} has also recently considered the relationship between the equality $g(a_1) = r$ and the Cohen-Macaulayness of $\gr_m(R)$.
\end{ex}

We get the following corollaries by combining Proposition \ref{mpure} and Theorem \ref{equivs}, and Proposition \ref{beta stuff} and Theorem \ref{equivs} respectively. The first contains special cases of \citep[Theorem 3.1(3)]{HKU} (or \citep[Theorem 1.5]{O}) and \citep[Theorem 3.9]{HKU}.

\begin{cor}With notation as in Theorem \ref{equivs},
\begin{enumerate}
\item $\gr_{\bar m}(\bar R)$ is Gorenstein $\iff$ $R$ is Gorenstein and $(x^{a_1}):m^n = (x^{a_1}) + m^{g-n+1}$ for $1\le n\le g$.
\item $\gr_m(R)$ is Gorenstein $\iff$ $R$ is Gorenstein and $(x^{a_1}):m^n = (x^{a_1}) + m^{r-n+1}$ for $1\le n\le r$.
\end{enumerate}
\end{cor}

\begin{cor}\label{hilbert} With notation as in Theorem \ref{equivs}, also let $H_{\gr_{\bar m}(\bar R)}(t) = \sum_{j=0}^g h_it^i$ be the Hilbert series of $\gr_{\bar m}(\bar R)$. If $R$ is Gorenstein, then $\sum_{j=0}^i h_j \ge \sum_{j=0}^i h_{g-j}$ for $0\le i\le g$ with equality holding for all $i$ if and only if $\gr_{\bar m}(\bar R)$ is Gorenstein.
\end{cor}


Some definitions are needed so that we can continue to work in the semigroup setting. The terminology in Definition \ref{m def}(1) makes reference to the property of the order function as stated in Lemma \ref{garcia}. Corollary \ref{msymiffspa} follows from Theorem \ref{equivs}.

\begin{ddef}\label{m def}\
\begin{enumerate}
\item $S$ is called {\em $M$-additive} if $\gr_m(R)$ is Cohen-Macaulay.
\item $S$ is called {\em $M$-symmetric} if $\gr_m(R)$ is Gorenstein.
\end{enumerate}
\end{ddef}

\begin{cor}\label{msymiffspa} A semigroup $S$ is $M$-symmetric if and only if it is symmetric, $M$-pure, and $M$-additive.
\end{cor}

The next task is to find $M$-pure semigroups. First consider the following list of $M$-additive semigroups.

\begin{prop}\label{grcm} \citep[Theorem 3.12]{B} If the semigroup $S$ satisfies one of the following conditions, then it is $M$-additive.
\begin{enumerate}
\item $S$ has embedding dimension 2.
\item $S$ is of maximal embedding dimension (i.e. $\nu = e$).
\item $S$ is of almost maximal embedding dimension (i.e. $\nu = e-1$) and $\#T<e-2$.
\item $S$ is generated by an arithmetic sequence.
\item $e\le 4$, except the case $S = <4,a_2,a_3>$ such that $a_3 = 3a_2-4$.
\item $S$ is symmetric and $\nu=e-2$.
\end{enumerate}
\end{prop}

This list is not exhaustive, but it provides many good examples. As the next proposition shows, many of these semigroups are also $M$-pure.

\begin{prop}\label{mpureadd} If the semigroup $S$ satisfies one of the following conditions, then it is $M$-pure and $M$-additive.
\begin{enumerate}
\item $S$ has embedding dimension 2.
\item $S$ is of maximal embedding dimension.
\item $S$ is symmetric and of almost maximal embedding dimension.
\item $S$ is generated by an arithmetic sequence.
\item $e\le 4$, except the case $S = <4,a_2,a_3>$ such that $S$ is not symmetric.
\end{enumerate}
\end{prop}

\begin{proof} By Proposition \ref{grcm} $S$ is $M$-additive.\\

(1) $\Ap(S) = \{0,a_2,2a_2,\dots,(a_1-1)a_2\}$. Clearly $ia_2 \pre_M  (a_1-1)a_2$ for all $0\le i\le e-1$. Thus $\maxap_M(S) = \{(a_1-1)a_2\}$, and $S$ is $M$-pure.

(2) $\Ap(S) = \{0,a_2,a_3,\dots,a_\nu\}$ and $\maxap_M(S) = \{a_2,a_3,\dots,a_\nu\}$. Since every element of $\maxap_M(S)$ has order 1, $S$ is $M$-pure.

(3) $\Ap(S) = \{0,a_2,a_3,\dots,a_\nu,a_i+a_j\}$ where $i+j = \nu+2$. Clearly $a_i \pre_M  a_i+a_j$ for all $2\le i\le \nu$. Thus $\maxap_M(S) = \{a_2+a_\nu\}$, and $S$ is $M$-pure.

(4) Let $e-1 = q(\nu-1) + m$ where $0\le m< \nu-1$. Also let $2\le n\le \nu$ and $0\le k\le q$. Then $\Ap(S) = \{0,a_n+ka_\nu\}$ where $2\le n+k(\nu-1) \le e$. For $u = a_n+ka_\nu\in \Ap(S)$, this representation is unique for $u$. Hence $\ord(a_n+ka_\nu) = k+1$. We consider two cases.

The first case is when $m=0$. Clearly $\{a_n+(q-1)a_\nu \st 2\le n\le \nu\} \subset \maxap_M(S)$. Since $(a_n +ka\nu) + (q-k-1)a_\nu = a_n + (q-1)a_\nu$ and $\ord(a_n +ka\nu) + \ord((q-k-1)a_\nu)) = (k+1) + (q-k-1) = q = \ord(a_n + (q-1)a_nu)$, we have $\{a_n+(q-1)a_\nu \st 2\le n\le \nu\} = \maxap_M(S)$. So $S$ is $M$-pure.

The second case is when $m>0$. Clearly $\{a_n+qa_\nu \st 2\le n\le m+1\} \subset \maxap_M(S)$. Since, for $k<q$, $\big(a_n + ka_\nu \big) + \big(a_{\nu - n+2} +(q-k-1)a_\nu \big) = a_2 + qa_\nu$ and $\ord(a_n + ka_\nu \big)) + \ord(a_{\nu - n+1} +(q-k-1)a_\nu) = (k+1) + (q-k) = q+1 = \ord(a_2 + qa_\nu)$, we have $\{a_n+qa_\nu \st 2\le n\le m+1\}= \maxap_M(S)$. So $S$ is $M$-pure.

(5) If $S = <4,a_2,a_3>$ and $S$ is not symmetric, then $\Ap(S) = \{0,a_2,2a_2,a_3\}$ or $\Ap(S) = \{0,a_2,a_3,2a_2\}$. Either way, $\maxap_M(S)=\{a_3,2a_2\}$. Since the orders of the elements in $\maxap_M(S)$ are not equal $S$ is not $M$-pure. The rest follows from (1), (2) and (3).
\end{proof}

\begin{cor}\label{grg} If the semigroup $S$ satisfies one of the following conditions, then $S$ is $M$-symmetric.
\begin{enumerate}
\item $S$ has embedding dimension 2.
\item $S$ is symmetric and of almost maximal embedding dimension.
\item $S$ is generated by an arithmetic sequence and $(e-2)/(\nu-1)$ is an integer.
\item $e\le 4$ and $S$ is symmetric.
\end{enumerate}
\end{cor}

\begin{proof} By Proposition \ref{mpureadd} and Corollary \ref{msymiffspa}, we need only show that $S$ is symmetric. This is assumed to be true for (2) and (4). For (1) and (3) we can use Proposition \ref{psym} since we know the Apery set of $S$ for these cases (see the proof of Proposition \ref{mpureadd}).
\end{proof}

In general there are no implications among the conditions symmetry, $M$-purity and $M$-additivity. In other words there exists a semigroup satisfying any combination of these conditions. Eight such examples with the least multiplicity and embedding dimension possible will be given. The next few lemmas will be helpful, some of which are interesting in their own right. These lemmas will give some idea of how the conditions are connected, at least for small multiplicities.

\begin{lem}\label{g and ev} Let $S$ be a semigroup. Then
\begin{enumerate}
\item $g(a_1) \le e-\nu+1$.
\item If $S$ is $M$-pure symmetric and $2<\nu<e-1$, then $g(a_1) \le e-2\nu+3$. Moreover, if $\nu = e-k$, then $k+3\le e \le 2k$.
\end{enumerate}
\end{lem}

\begin{proof} (1) We have $\#\Ap(S) = e$ and the orders of the elements of $\Ap(S)$ are the consecutive integers $0,1,\dots,g(a_1)$. Since the minimal generators all have order 1, $g(a_1) \le e-1 - (\nu-2) = e - \nu +1$.

(2) If $S$ is $M$-pure symmetric and $2<\nu<e-1$, then there are $\nu-1$ elements of $\Ap(S)$ with order $g(a_1)-1$ that are distinct from the minimal generators. Thus, similar to (1), $g(a_1) \le e-1 - 2(\nu-2) = e - 2\nu +3$. For the second statement notice that $2 < \nu = e-k$ and so $k+3 \le e$. Moreover, $g(a_1) \ge 3$. So $g(a_1) \le e-2\nu+3 \Rightarrow e \le 2k+3 -g(a_1) \le 2k$.
\end{proof}

\begin{lem}\label{almost max} If $S$ is of almost maximal embedding dimension, then the following are equivalent.
\begin{enumerate}
\item $S$ is $M$-pure.
\item $S$ is symmetric.
\item $S$ is $M$-symmetric.
\end{enumerate}
\end{lem}

\begin{proof} (1) $\Rightarrow$ (3) Since all of the nonzero elements of $\Ap(S)$ are minimal generators (with order 1) with the exception of one element $w$ with order 2, $\maxap_M(S) = \{w\}$. Therefore by Proposition \ref{mpsym}, $S$ is $M$-symmetric.

(3) $\Rightarrow$ (2) Clear.

(2) $\Rightarrow$ (1) Proposition \ref{mpureadd}.
\end{proof}

\begin{lem}\label{symax embed} A semigroup $S$ is symmetric and of maximal embedding if and only if $e=2$.
\end{lem}

\begin{proof} Let $S$ by symmetric and of maximal embedding. By Proposition \ref{psym}, $\#\maxap(S) = 1$. On the other hand, if $S$ is of maximal embedding dimension, then $\#\maxap(S) = e-1$. Thus $e=2$. The converse is clear.
\end{proof}

\begin{lem}\label{3lem} Let $S$ be a symmetric semigroup with embedding dimension 3.
\begin{enumerate}
\item If $e=5$, then $S$ is not $M$-pure.
\item If $e=6$, then $S$ is $M$-pure.
\item If $e=7$, then $S$ is not $M$-pure.
\end{enumerate}
\end{lem}

\begin{proof} (1) This follows from Lemma \ref{g and ev}(2).

(2) $3\le \ord(w_{e-1}) =g(a_1)\le 4$ by Propositions \ref{psym}, \ref{gAp}, and Lemma \ref{g and ev}(1). Suppose that $\ord(w_{e-1}) = 4$. Then $w_{e-1}$ is equal to $4a_2$, $3a_2 + a_3$, $2a_2+2a_3$, $a_2+3a_3$, or $4a_3$. All of these cases lead to a contradiction. If $w_{e-1}=4a_2$, then $\Ap(S) = \{0,a_2,\underline{\ \ },\underline{\ \ },3a_2,4a_2\}$ where $a_3$ and $2a_2$ fill the blanks. But then $2a_2 + a_3 = 4a_2 \Rightarrow a_3 = 2a_2$. If $w_{e-1}=3a_2+a_3$, then $\Ap(S) = \{0,a_2,\underline{\ \ },\underline{\ \ },2a_2+a_3,3a_2+a_3\}$ where $a_3$, $2a_2$, and $a_2+a_3$ fill the blanks. If $w_{e-1}=2a_2+2a_3$, then $\Ap(S) = \{0,a_2,\underline{\ \ },\underline{\ \ },a_2+2a_3,2a_2+2a_3\}$ where $a_3$, $2a_3$, and $a_2+a_3$ fill the blanks. If $w_{e-1}=a_2+3a_3$, then $\Ap(S) = \{0,a_2,\underline{\ \ },\underline{\ \ },3a_3,a_2+3a_3\}$ where $a_3$, $2a_3$, $a_2+a_3$, and $a_2+2a_3$ fill the blanks. Lastly if $w_{e-1}=4a_3$, then $\Ap(S) = \{0,a_2,a_3,2a_3,3a_3,4a_3\}$. But then $2a_2 \equiv 4a_3 \mod 6$ and $2a_2 < 4a_3$. So we have $\ord(w_{e-1}) = 3$ and the orders of the elements of $\Ap(S)$ are $0,1,1,2,2,3$ (not necessarily in that order). By Proposition \ref{beta stuff} $S$ is $M$-pure.

(3) The Ap\'ery set of $S$ can be either $\{0,a_2,a_3,2a_2,2a_3,3a_2,4a_2=3a_3\}$ ($<7,9,12>$) $\{0,a_2,a_3,2a_2,a_2+a_3,3a_2=2a_3,4a_2=a_2+2a_3\}$ ($<7,8,12>)$, or $\{0,a_2,2a_2,a_3,3a_2,4a_2,5a_2=2a_3\}$ ($<7,8,20>$). The orders of the elements are $\{0,1,1,2,2,3,4\}$, $\{0,1,1,2,2,3,4\}$, $\{0,1,2,1,3,4,5\}$ respectively. By Proposition \ref{beta stuff}, $S$ is not $M$-pure.

To see that these are the only possibilities, Let $\Ap(S) = \{0,a_2,\underline{\ \ },\underline{\ \ },\underline{\ \ },\underline{\ \ },\underline{\ \ }\}$. If $a_3$ fills any of the last three blanks, then $\Ap(S) = \{0,a_2,2a_2,3a_2,4a_2,5a_2,6a_2\}$. So $a_3$ is a multiple of $a_2$, which is a contradiction. The only other possibility that requires some justification to exclude is $\Ap(S) = \{0,a_2,a_3,a_2+a_3,2a_2+a_3,a_2+2a_3,2a_2+2a_3\}$. Since $2a_2 \not\in \Ap(S)$, $2a_2 \equiv a_3 \mod 7$, and since $2a_3 \not\in \Ap(S)$, $2a_3 \equiv a_2 \mod 7$. Thus $a_2 \equiv 4a_2 \mod 7$, which implies that $a_2 \equiv 0 \mod 7$. This is a contradiction.
\end{proof}

\begin{ques} Let $S$ be a symmetric semigroup with embedding dimension 3, in which case $e\ge 4$. It can be shown by combining Proposition \ref{mpureadd}, Lemma \ref{symax embed} and Lemma \ref{3lem} that if $4 \le e\le 7$, then $S$ is $M$-pure if and only if $e$ is even. Does this hold for all $e\ge 4$?
\end{ques}

The next lemma is an immediate consequence of \citep[Theorem 3.14]{B}.

\begin{lem}\label{bar and fro} Let $a_2=a_1+1$. Then $S$ is $M$-additive if and only if $0\le w_i - a_1\ord(w_i)\le e-1$ for all $0\le i\le e-1$.
\end{lem}

\begin{exs}\label{8exs} These examples show that in general there are no implications among the conditions symmetry, $M$-purity, and $M$-additivity. In addition each example has the least multiplicity and embedding dimension possible.
\begin{enumerate}
\item $S=<2,3>$ is symmetric, $M$-pure, and $M$-additive
\item $S=<3,4,5>$ $M$-pure and $M$-additive, but not symmetric
\item $S=<4,5,7>$ is $M$-additive, but not symmetric or $M$-pure
\item $S=<4,5,11>$ is not symmetric, $M$-pure, or $M$-additive
\item $S=<5,6,9>$ is symmetric and $M$-additive, but not $M$-pure
\item $S=<5,6,13>$ is $M$-pure, but not symmetric or $M$-additive
\item $S=<6,7,15>$ is symmetric and $M$-pure, but not $M$-additive
\item $S=<7,8,20>$ is symmetric, but not $M$-pure or $M$-additive
\end{enumerate}
\end{exs}

\begin{proof} That the embedding dimension cannot be lowered is clear for all the examples by Corollary \ref{grg}.

(1) Corollaries \ref{msymiffspa} and \ref{grg}. That the multiplicity is minimal is clear.

(2) Proposition \ref{mpureadd} and Lemma \ref{symax embed}. The multiplicity is minimal since every semigroup with $e=2$ is $M$-symmetric by Corollary \ref{grg}.

(3) It is easy to check that $S$ is not symmetric, so we apply Lemmas \ref{almost max} and \ref{bar and fro}. The multiplicity is minimal since every semigroup with $2\le e\le 3$ is $M$-pure by Proposition \ref{mpureadd}.

(4) Same as (3).

(5) It is easy to check that $S$ is symmetric, so we apply Proposition \ref{grcm} and Lemma \ref{3lem}. The multiplicity is minimal since every symmetric semigroup with $2\le e\le 4$ is $M$-pure by Proposition \ref{mpureadd}.

(6) It is easy to check that $S$ is $M$-pure, but not symmetric. By Lemma \ref{bar and fro} $S$ is not $M$-additive. The multiplicity is minimal since every $M$-pure semigroup with $2\le e\le 4$ is $M$-additive by Proposition \ref{mpureadd}, using Lemma \ref{almost max} as well for the case $S=<4,a_2,a_3>$.

(7) It is easy to check that $S$ is symmetric, so we apply Lemmas \ref{3lem} and \ref{bar and fro}. The multiplicity is minimal since every $M$-pure symmetric semigroup with $2\le e\le 5$ is $M$-additive by Proposition \ref{mpureadd}, using Proposition \ref{grcm} as well for the case $S=<5,a_2,a_3>$.

(8) It is easy to check that $S$ is symmetric, so we apply Lemmas \ref{3lem} and \ref{bar and fro}. The multiplicity is minimal since every symmetric semigroup with $2\le e\le 6$ is either $M$-pure or $M$-additive by Proposition \ref{grcm}, using Lemma \ref{3lem} for the case $S=<6,a_2,a_3>$.
\end{proof}


\section{The Minimal Goto Number of a Numerical Semigroup Ring}\label{minsect}

Now we consider the minimal Goto number of a semigroup $S$, which is denoted by $\tau$. Recall that $\tau$ is also the minimal Goto number of $R=k[[S]]$ \citep[Theorem 4.1]{HS}. An open problem is to determine when $\tau = g(a_i)$ for some minimal generator $a_i$ of $S$. In particular, when is $\tau = g(a_1)$? It was shown by \citet[Theorem 5.10]{HS} that $g(a_1) = \tau$ if $\nu =2$. On the other hand, $S =<7,11,20>$ and $S=<11,14,21>$ are examples due to \citet{S} of semigroups for which $\tau < g(a_i)$ for all $i$. Notice that the latter is symmetric and $M$-additive.

\begin{ddef} Let $S$ be a semigroup with $\Ap(S) = \{w_0,w_1,\dots,w_{e-1}\}$.
\begin{eqnarray*}
\delta_i &=& \max\{\ord(w_i) + \ord(w) \st w\in S \text{ and } w_i + w\in \Ap(S)\}\\
\gamma_i &=& \max\{\ord(w_i) + \floor{w} \st w\in S \text{ and } w_i + w\in \Ap(S)\}\\
\delta &=& \min\{\delta_i \st 0\le i\le e-1\}\\
\gamma &=& \min\{\gamma_i \st 0\le i\le e-1\}
\end{eqnarray*}
\end{ddef}

For a semigroup $S$ with corresponding ring $(R,m)=k[[S]]$, one has the inequalities
\begin{equation}\label{ineq} \tag{$\ast$}
\delta \le \gamma \le \ord(C) \le \tau \le g(a_1) \le r
\end{equation}

\noindent where $\ord(C) = \min\{\ord(t) \st 0\ne t\in C\}=\min\{\ord(f+1),\dots,\ord(f+a_1)\}$ is the $m$-adic order of $C=R:k[[x]]$ and $r = r(m)$ is the reduction number $m$ (with respect to $x^{a_1}R$).

Indeed by Lemma \ref{about ord}, $\delta \le \gamma$. Now consider $\gamma \le \ord(C)$. For $1\le \alpha\le a_1$ we have $f+\alpha = w_i +ka_1$ for some $w_i \in \Ap(S)$ and $k\ge 0$. Let $w_j\in S$ such that $w_i +w_j =w_n\in \Ap(S)$ and $\ord(w_i) + \sfloor{w_j}\ge \gamma$. Then $f+\alpha + w_j = w_n + ka_1$. There are two cases.

First if $f<w_n$, then $w_n = f+\beta$ for some $1\le \beta\le a_1$. We set $\beta' = a_1-\beta$. Then we have $\alpha +w_j +\beta' = (k + 1)a_1$. It follows that $\alpha +w_j - a_1\sfloor{w_j} +\beta' = (k+1-\sfloor{w_j})a_1$ is a positive multiple of $a_1$, and so $k\ge \sfloor{w_j}$. Now we have $\ord(f+\alpha) \ge \ord(w_i) + k \ge\ord(w_i) + \sfloor{w_j} \ge \gamma$.

The second case is when $f> w_n$. We have $f+\alpha + w_j + v_{-\wh n} = (1+\sfloor{w_n} +\sfloor{v_{-\wh n}} + k)a_1$ by Lemma \ref{apnot} (Recall Definition \ref{biject}). It follows that $f-a_1\sfloor{f} +\alpha + w_j-a_1\sfloor{w_j} + v_{-\wh n}-a_1\sfloor{v_{-\wh{n}}}= (1+\sfloor{w_n} + k-\sfloor{f}-\sfloor{w_j})a_1$ is a positive multiple of $a_1$, and so $k  \ge \sfloor{w_j} + \sfloor{f}-\sfloor{w_n}$. Now we have $\ord(f+\alpha) \ge \ord(w_i) + k \ge \ord(w_i) + \sfloor{w_j} + \sfloor{f}-\sfloor{w_n}\ge \ord(w_i) + \sfloor{w_j} \ge \gamma$. This establishes the second inequality.

Now consider $\ord(C) \le \tau$. If $(\sigma(1),\dots,\sigma(a_1))$ is the Goto vector of $S$, then $\ord(f+\alpha) \le \sigma(\alpha)$ for $1\le \alpha \le a_1$. Thus $\ord(C)=\min\{\ord(f+1),\dots,\ord(f+a_1)\} \le \min\{\sigma(1),\dots,\sigma(a_1)\} = \tau$, and we have the third inequality.

The fourth inequality $\tau \le g(a_1)$ is trivial, and the last inequality $g(a_1) \le r$ follows easily from Definition \ref{goto def} since $x^{a_1}R:m^{r+1} = R$. The following examples show that all of these inequalities can be strict.

\begin{exs} \
\begin{enumerate}
\item Let $S=<5,8,12>$, then $\delta = 2$ and $\gamma = \ord(C) = \tau = g(a_1) = r = 3$.
\item Let $S=<4,7,9>$, then $\delta = \gamma=1$ and $\ord(C) = \tau = g(a_1) = r = 2$.
\item Let $S=<5,6,14>$, then $\delta = \gamma= \ord(\mathcal C) = 1$, $\tau=2$, $g(a_1)=3$, and $r=4$.
\end{enumerate}
\end{exs}

The next proposition gives sufficient conditions for some of the inequalities in (\ref{ineq}) to be equalities.

\begin{prop}\label{someeq} Let $S$ be a semigroup and $\Ap(S) = \{w_0,w_1\dots,w_{e-1}\}$.
\begin{enumerate}
\item If $S$ is symmetric, then $\ord(C) = \tau$.
\item If $S$ is $M$-additive, then $g(a_1)=r$.
\item If $S$ is $M$-additive symmetric, then $\gamma = \ord(C) =\tau\le g(a_1) = r$.
\item If $S$ is symmetric, then $\ord(w_i) = \sfloor{w_i}$ for some $w_i \in\Ap(S)$ such that $\delta_{e-1-i}=\delta$ $\iff$ $\delta = \gamma$.
\end{enumerate}
\end{prop}

\begin{proof} (1) follows from Proposition \ref{consym}, and (2) follows from Theorem \ref{equivs}(2). (3) is proven in Corollary \ref{gam=tau} in the next section. For (4), we have $\gamma \le \ord(w_{e-1-i})+\sfloor{w_i} = \ord(w_{e-1-i})+\ord(w_i) =\delta$. Thus $\delta = \gamma$. Conversely, there exists $w_i \in\Ap(S)$ such that $\delta \le \ord(w_{e-1-i})+\ord(w_i) \stackrel{\ref{about ord}}{\le} \ord(w_{e-1-i})+\sfloor{w_i} = \delta$. Thus $\ord(w_i)= \sfloor{w_i}$ and $\delta_{e-1-i}=\delta$.
\end{proof}

The integer $\delta$ was introduced because it provides a lower bound for the minimal Goto number of $S$ (in fact for the $m$-adic order of the conductor ideal $C$), and because of its connection with the $M$-purity of $S$. The latter is the content of the next proposition.

\begin{thm}\label{puredel} Let $S$ be a semigroup.
\begin{enumerate}
\item $S$ is $M$-pure if and only if $\delta = g(a_1)$.
\item $S$ is $M$-pure and $M$-additive if and only if $\delta = r$.
\item $S$ is $M$-symmetric if and only if $S$ is symmetric and $\delta =r$.
\end{enumerate}
\end{thm}

\begin{proof} (1) If $S$ is $M$-pure, then for any $w_i \in \Ap(S)$ there exists $w_j \in \Ap(S)$ such that $w_i+w_j \in \Ap(S)$ and $\ord(w_i)+\ord(w_j) = g(a_1)$. Thus $\delta_i = g(a_1)$ for all $0\le i\le e-1$. Conversely, let $\delta = g(a_1)$. Then for any $w_i\in \Ap(S)$ there exists $w_j \in \Ap(S)$ such that $w_i+w_j =w\in\Ap(S)$ and $\ord(w_i)+\ord(w_j)=\ord(w) = g(a_1)$. Thus $S$ is $M$-pure.

(2) If $S$ is $M$-pure and $S$ is $M$-additive, then by (1) $\delta = g(a_1)$ and by Theorem \ref{equivs} $g(a_1) = r$. Conversely if $\delta = r$, then in fact $\delta = g(a_1) = r$. By (1) $S$ is $M$-pure and hence by Theorem \ref{equivs} $\gr_m(R)$ is Cohen-Macaulay.

(3) This follows from Corollary \ref{msymiffspa} and (2).
\end{proof}

\begin{cor} Let $S$ be a semigroup with corresponding ring $R=k[[S]]$. If $\gr_{\bar m}(\bar R)$ is Gorenstein, then $\delta = \gamma = \ord(C) = \tau = g(a_1)$.
\end{cor}

\begin{proof} Theorems \ref{equivs} and \ref{puredel}.
\end{proof}

\begin{cor}\label{cor49} For the following semigroups we have $\delta =\gamma =\ord(C)=\tau= g(a_1)=r$.
\begin{enumerate}
\item $S$ is of maximal embedding dimension.
\item $S$ is generated by an arithmetic sequence.
\item $S$ is $M$-symmetric ($\gr_m(R)$ is Gorenstein).
\end{enumerate}
\end{cor}

\begin{proof} Corollary \ref{msymiffspa}, Proposition \ref{mpureadd}, and Theorem \ref{puredel}.
\end{proof}

We have seen that if $\gr_m(R)$ is Gorenstein, then $\tau=g(a_1)$. The next proposition deals with the converse statement.

\begin{prop}\label{cor50} Let $R=k[[S]]$. Then $\gr_m(R)$ is Gorenstein if and only if the following conditions hold.
\begin{enumerate}
\item $R$ is Gorenstein
\item $\gr_m(R)$ is Cohen-Macaulay
\item $\tau=g(a_1)$
\item $\delta=\gamma$
\end{enumerate}
\end{prop}


\begin{proof} Proposition \ref{someeq} and Corollary \ref{cor49}.
\end{proof}

\begin{ex} The semigroup $S=<5,8,12>$ satisfies the first three conditions of Proposition \ref{cor50}, but not the fourth. Thus $\gr_m(R)$ is not Gorenstein and $\delta < \gamma = \ord(C) = \tau= g(a_1) =r$.
\end{ex}

\begin{rmk} Working with the defining ideal of a numerical semigroup ring $R=k[[S]]$, i.e., the kernel of the surjection $S=k[[t_1,\dots,t_{\nu}]] \twoheadrightarrow \Rgen$ mapping $t_i$ to $x^{a_i}$, \citet[Corollary 2.6]{S} was able to show that $\ord(C) = \tau = g(a_1) = r$ when $\gr_m(R)$ is Gorenstein and $\nu \le 4$.
\end{rmk}

The hypothesis that $S$ is $M$-pure can be weakened and still have $\tau = g(a_1)$ if it is assumed that $a_2$ is sufficiently large.

\begin{prop}\label{puret=g} Let $S$ be pure and $\sfloor{a_2} \ge g(a_1)-1$, then $\gamma = \ord(C) =\tau = g(a_1)$.
\end{prop}

\begin{proof} It suffices to show that $\gamma = g(a_1)$. Let $w_i \in \Ap(S)$. If $w_i = 0$, choose $w\in \maxap(S)$. Then $\gamma_i = \sfloor{w_{e-1}} \ge \sfloor{w} \ge \ord(w) =g(a_1)$. If $0\ne w_i \not\in \maxap(S)$, then there exists some $w_j\in \Ap(S)$ such that $\gamma_i = \ord(w_i) + \sfloor{w_j} \ge \ord(w_i) +\sfloor{a_2} \ge 1 + (g(a_1) -1) = g(a_1)$. Lastly if $w_i \in \maxap(S)$, then $\gamma_i = \ord(w_i) = g(a_1)$. Therefore $\gamma = g(a_1)$.
\end{proof}

By Proposition \ref{mpureadd} and Theorem \ref{puredel}, $\tau = g(a_1)$ for every symmetric semigroup with $e\le 4$ (in fact in Theorem \ref{ele4} we show that $S=<4,5,7>$ is the only semigroup with $e\le 4$ such that $\tau < g(a_1)$). The next theorem concerns symmetric semigroups, in which case more can be proven using Proposition \ref{puret=g}.

\begin{thm}\label{except} The equality $\tau=g(a_1)$ holds for every symmetric semigroup with $e\le 6$ except the following semigroups.
\begin{enumerate}
\item $S=<5,6,9>$
\item $S=<6,7,10,11>$
\end{enumerate}
\end{thm}

\begin{proof} We want to find all the symmetric semigroups with $e\le 6$ such that $\tau < g(a_1)$. By Proposition \ref{mpureadd}, Lemma \ref{3lem}, and Theorem \ref{puredel} we need only consider the cases $S = <5,a_2,a_3>$ and $S = <6,a_2,a_3,a_4>$.

If $S = <5,a_2,a_3>$, then $\Ap(S) = \{0,a_2,a_3,2a_2,3a_2 = 2a_3\}$. Also $a_2\le 9$ by Proposition \ref{puret=g}. The only semigroups that satisfy these requirements are $S = <5,6,9>$ and $S = <5,8,12>$. Now it is easy to check that $\ord(C) = \tau = 2 < 3 = g(a_1)$ for $S = <5,6,9>$, and that $\ord(C) = \tau = 3 = g(a_1)$ for $S = <5,8,12>$.

If $S = <6,a_2,a_3,a_4>$, than $\Ap(S) = \{0,a_2,a_3,a_4,2a_2,3a_2=a_3+a_4\}$. Also $a_2\le 11$ by Proposition \ref{puret=g}. The only semigroups that satisfy these requirements are $S = <6,7,10,11>$, $S = <6,11,13,20>$, and $S = <6,11,14,19>$. Again it is easy to check that $\ord(C) = \tau = 2 < 3 = g(a_1)$ for $S = <6,7,10,11>$, and that $\ord(C) = \tau = 3 = g(a_1)$ for the others.
\end{proof}

So far in this section, many of the results have been concerned with the equality $\tau=g(a_1)$. By Proposition \ref{mpureadd}, $S$ is $M$-pure and $\tau=g(a_1)$ if $\nu = 2$ or $\nu = e$. Also by this proposition $S$ is $M$-pure and $\tau=g(a_1)$ if it is symmetric and $e\le 4$. The next two examples show that this is the best we can do when we are considering all semigroups with a fixed multiplicity and embedding dimension, or a fixed multiplicity when $S$ is symmetric.

\begin{ex}\label{ex41} If $2 < \nu < e$, then the semigroup $S=<e,e+1,\dots,e+\nu-2,2e-1>$ has embedding dimension $\nu$, multiplicity $e$, and $\tau < g(a_1)$. The only claim that requires some justification is that $\tau < g(a_1)$.

There are two kinds of elements in $\Ap(S)$, so we write $\Ap(S) = A\cup B$ where $A = \{(e+h) + k(e+\nu-2) \st 1\le h\le \nu-2$, $0\le k\}$ and $B=\{l(2e-1)\st 1\le l\}$. Likewise there are two kinds of elements in $\Ap(S;1)$, so we write $\Ap(S;1) = A'\cup B'$ where $A' = \{l(2e-1)\st 1\le l\}$ and $B'=\{l(2e-1)+e\st 1\le l\}$. Now let $L\ge 1$ be the largest integer $m$ such that $\{me, me+1,\dots,me +e-1\} \not\subset S$. We consider two cases. In both cases we have $\tau = \min\{\sigma(1),\sigma(2),\dots,\sigma(e)\} \le \sigma(1) < \sigma(e) = g(a_1)$, where $(\sigma(1),\sigma(2),\dots,\sigma(e))$ is the Goto vector of $S$.

If $L = 2n-1$, then $\max\{\ord(w) \st w\in A\} = L+1 = 2n$ and $\max\{\ord(w) \st w\in B\} = n$. Thus $\sigma(e) = 2n$. On the other hand $\max\{\ord(w) \st w\in A'\} = n$ and $\max\{\ord(w) \st w\in B'\} = n$. Thus $\sigma(1) = n$.

If $L = 2n$, then $\max\{\ord(w) \st w\in A\} = L+1 = 2n+1$ and $\max\{\ord(w) \st w\in B\} = n+1$. Thus $\sigma(e) = 2n+1$. On the other hand $\max\{\ord(w) \st w\in A'\} = n$ and $\max\{\ord(w) \st w\in B'\} = n+1$. Thus $\sigma(1) = n+1$.
\end{ex}

\begin{ex}\label{ex42} If $e \ge 5$, then the semigroup $S=<e,e+1,e+4,e+5,\dots,2e-1>$ is symmetric, has multiplicity $e$, and $\tau < g(a_1)$. We have $$S = \{0,e,e+1,e+4,e+5,\dots,2e-1,2e,2e+1,2e+2,2e+4,\rightarrow\}$$ Clearly the multiplicity of $S$ is $e$. Moreover, $f=2e+3$ and in the set $\{0,1,\dots,2e+3\}$ exactly half are in $S$ and half are not in $S$. Thus $S$ is symmetric. Lastly $\tau = \min\{\ord(f+1),\ord(f+2), \dots,\ord(f+a_1)\} \le \ord(f+1) = \ord(2e+4) = 2 < 3 = \ord(3e+3) = \ord(f+a_1) = g(a_1)$.
\end{ex}

In the next proposition $\tau$ is related to the multiplicity and embedding dimension in special cases. 

\begin{prop}\label{tev} Let $S$ be a semigroup.
\begin{enumerate}
\item $\tau \le e-\nu+1$
\item $\tau \le e-2\nu+3$ if $S$ is $M$-pure symmetric and $2<\nu<e-1$.
\item $\tau = e-\nu+1 = \left\{ \begin{array} {c@{\quad}l} e-1 & \text{if } \nu = 2\\ 2 & \text{if } \nu = e-1 \text{ and } S \text{ is symmetric}\\ 1 & \text{if } \nu = e \end{array} \right.$
\item $\tau =\ceil[\nu-1]{e-1}$ if $S$ is generated by an arithmetic sequence.
\end{enumerate}
\end{prop}

\begin{proof} (1) and (2) follow from Lemma \ref{g and ev}. For (3) and (4), we have $\tau = g(a_1)$ by Corollary \ref{cor49}. In the proof of Proposition \ref{mpureadd} we have the Ap\'ery sets for these semigroups, from which we can determine $g(a_1)$ by Proposition \ref{gAp}. The result follows.
\end{proof}

The next proposition gives a class of numerical semigroup rings $R$ for which every parameter ideal $Q$ of $R$ has the same Goto number. It is important to note that this result is not restricted to the monomial parameter ideals. This extends a result of \citet[Theorem 3.3]{GKM}.

\begin{prop} Let $S$ be a semigroup and $R=k[[S]]$ its corresponding ring. If $S$ is generated by consecutive integers, then $g(Q) =  \sceil[\nu-1]{e-1}$ for every parameter ideal $Q$ of $R$.
\end{prop}

\begin{proof} Let $e-1 = q(\nu-1) + m$ where $0\le m< \nu-1$. By Remark \ref{3.2} and Proposition \ref{tev}, $\sceil[\nu-1]{e-1} = \tau \le g(Q) \le \sceil{f}$ for {\em every} parameter ideal $Q$ of $R$. Thus it suffices to show that $\sceil{f} = \sceil[\nu-1]{e-1}$. Let $2\le n\le \nu$ and $0\le k\le q$. Then $\Ap(S) = \{0,a_n+ka_\nu\}$ where $2\le n+k(\nu-1) \le e$. If $m = 0$, then $f = qa_\nu -a_1 = qa_1-1$. So $\sceil{f} = \sceil{qa_1-1} = q - \sfloor{1} = q = \sfloor[\nu-1]{e-1} = \sceil[\nu-1]{e-1}$. If $m>0$, then $f = a_{m+1} + qa_\nu -a_1= (q+1)a_1 -1$. So $\sceil{f} = \sceil{(q+1)a_1-1} = q+1 - \sfloor{1} = q+1 = \sfloor[\nu-1]{e-1} +1 = \sceil[\nu-1]{e-1}$.
\end{proof}


\section{Computing the Goto Numbers of a Numerical Semigroup}\label{comp}

In this section we consider computing all the Goto numbers of a semigroup $S$ in terms of its minimal generators. As an intermediate step, we will compute them in terms of the elements of the Ap\'ery set of $S$. Recall that by Definition \ref{biject}, $\Ap(S) = \{w_0,w_1,\dots,w_{e-1}\} = \{v_0,v_1,\dots,v_{e-1}\}$ and $w_i = v_{\wh \imath} \equiv \wh \imath \mod e$ where $1\le \wh \imath\le e$. Moreover, we will allow the subscript of $v_n$ to be any integer by agreeing that $v_{n} = v_{m}$ if and only if $n \equiv m \mod e$. We begin by noting that $\wh \imath = w_i -(\sceil{w_i}-1)a_1$, and so $\A(S) = \{ \wh \imath \st 0\le i\le e-1\} = \{ w_i -(\sceil{w_i}-1)a_1 \st 0\le i\le e-1\}$. With this representation of $\wh \imath$, $\A(u)$ for $u\in S$ is determined in Lemma \ref{au}. We make use of Lemma \ref{apnot}, and properties of ceiling and floor functions throughout this section. For more about ceiling and floor functions see \citep[Chapter 3]{GKP}.

\begin{lem}\label{au} Let $S$ be a semigroup and $u\in S$. Choose $w_h$ to be the largest element of $\Ap(S)$ such that $w_h< u$, and write $u=w_p+ka_1$ where $w_p\in \Ap(S)$ and $k\ge0$. Then $$\A(u) =\left\{v_{\wh p-\wh \jmath} -\left(\ceil{v_{\wh p-\wh \jmath}}-1\right)a_1 \st 0\le j\le h\right\}$$
\end{lem}

\begin{proof} First we show that $\A(u) = \{u-w_q-(\sceil{u-w_q}-1)a_1\st w_q<u\}$. Indeed let $\alpha\in \A(u)$. Then $\alpha = u-u'$ for some $u'\in S$, and we can write $u'=w_q+(\sceil{u-w_q}-1)a_1$ where $w_q\in \Ap(S)$. Thus $\alpha = u-w_q-(\sceil{u-w_q}-1)a_1$ where $w_q<u$. Conversely, if $w_q<u$, then set $\alpha = u-w_q-(\sceil{u-w_q}-1)a_1$. We have $1\le \alpha\le a_1$ and $u-\alpha \in S$, thus $\alpha \in \A(u)$.

Now we have
\begin{eqnarray*}
u-w_q-\left(\ceil{u-w_q}-1\right)a_1 &=& w_p+ka_1-w_q-\left(\ceil{w_p+ka_1-w_q}-1\right)a_1\\
&=& w_p-w_q-\left(\ceil{w_p-w_q}-1\right)a_1\\
&=& v_{\wh p-\wh q} -ta_1-\left(\ceil{v_{\wh p-\wh q}-ta_1}-1\right)a_1 \text{ where } t\ge0\\
&=& v_{\wh p-\wh q}-\left(\ceil{v_{\wh p-\wh q}}-1\right)a_1\\
\end{eqnarray*}

\noindent and so $\A(u) =\{v_{\wh p-\wh \jmath} -(\sceil{v_{\wh p-\wh \jmath}}-1)a_1 \st 0\le j\le h\}$.
\end{proof}

\begin{cor}\label{Acor} Let $S$ be a semigroup and $u \in S$. Then
\begin{enumerate}
\item $\#\A(u) = i$ $\iff$ $w_{i-1}$ is the largest element of $\Ap(S)$ strictly smaller than $u$.
\item $w_i = \max\{u\in S\st \#\A(u) = i\}$ for $0\le i\le e-1$.
\end{enumerate}
\end{cor}

Let $S=\{\lambda_0,\lambda_1,\lambda_2,\dots\}$ be an enumeration of $S$ and let $\pi_i = \#\A(\lambda_i)$. Corollary \ref{Acor} shows that the sequence $(\pi_i)$ has the structure $$0,1,\dots,1,2,\dots,2,\cdots,a_1-1,\dots,a_1-1,a_1,\dots$$ and the jumps occur at elements of $\Ap(S)$.

Lemma \ref{gs} gives a workable description of the elements of the Goto set of a semigroup. Recall that the Goto set of $S$ is the Goto vector taken as a set. The following proposition expresses the Goto number of a nonzero element in terms of the Ap\'ery set. If $S$ is symmetric this can be refined, and even more so if $S$ is $M$-additive symmetric. Proposition \ref{gu5} is the intermediate step referred to at the beginning of Section 5.

\begin{lem}\label{gs} Let $S$ be a semigroup and $gs(S)=\{\sigma(\,\wh \imath\,)\st 0\le i\le e-1\}$. Then
\begin{eqnarray*}\sigma(\,\wh \imath\,) &=& \max\left\{\ord\left(w+w_i -a_1\ceil{w_i}\right) \st w\in \maxap(S)\right\}\\
&=& \max\left\{\ord\left(w-v_{-\wh \imath} + a_1\floor{v_{-\wh \imath}}\right)\st w\in \maxap(S)\right\}
\end{eqnarray*}
\end{lem}

\begin{proof} By definition $\sigma(\,\wh \imath\,) = \max\{\ord\big(p+\wh \imath)\st p\in T\}$. But $p+\wh \imath = p +w_i -(\sceil{w_i}-1)a_1 = (p+a_1) + w_i -a_1\sceil{w_i}$, and by Lemma \ref{maxminchar} $\{p+a_1\st p\in T\} = \maxap(S)$. This shows the first equality. For the second, notice that $w+w_0 -a_1\sceil{w_0} = w-v_{-\wh 0} + a_1\sfloor{v_{-\wh 0}} = w$. For $1\le i\le e-1$, $w+w_i -a_1\sceil{w_i}=w+w_i -a_1\sfloor{w_i} -a_1 = w-v_{-\wh \imath} + a_1\sfloor{v_{-\wh \imath}}$ by Lemma \ref{apnot}.
\end{proof}

\begin{prop}\label{gu5} Let $S$ be a semigroup and $0\ne u\in S$. Set $u = w_p +ka_1$ for some $w_p \in \Ap(S)$ and $k\ge0$, and let $w_h$ be the largest element of $\Ap(S)$ strictly smaller than $u$. Then
\begin{eqnarray*}
g(u) &=& \min_{0\le j\le h} \left\{\, \max\left\{\ord\left(w+v_{\wh p-\wh \jmath} - a_1\ceil{v_{\wh p-\wh \jmath}}\right) \st w\in \maxap(S)\right\}\, \right\}\\
&=& \min_{0\le j\le h} \left\{\, \max\left\{\ord\left(w-v_{\wh \jmath-\wh p} + a_1\floor{v_{\wh \jmath-\wh p}}\right) \st w\in \maxap(S)\right\}\, \right\}
\end{eqnarray*}
\noindent If $S$ is symmetric this reduces to
\begin{eqnarray*}
g(u) &=& \min_{0\le j\le h} \left\{\ord\left(w_{e-1}+v_{\wh p-\wh \jmath} - a_1\ceil{v_{\wh p-\wh \jmath}}\right) \right\}\\
&=& \min_{0\le j\le h} \left\{\ord\left(w_{e-1}-v_{\wh \jmath-\wh p} + a_1\floor{v_{\wh \jmath-\wh p}}\right) \right\}\\
&=& \min_{0\le j\le h} \left\{\ord\left(v_{\wh{e-1} + \wh p -\wh \jmath} + a_1\floor{v_{\wh \jmath-\wh p}}\right) \right\}
\end{eqnarray*}
\noindent and if $S$ is $M$-additive symmetric we further have
\begin{eqnarray*}
g(u) &=& \min_{0\le j\le h} \left\{\ord\left(v_{\wh{e-1} + \wh p -\wh \jmath}\right) + \floor{v_{\wh \jmath-\wh p}} \right\}
\end{eqnarray*}
\end{prop}

\begin{proof} The first two equalities follow from Lemmas \ref{au} and \ref{gs}. When $S$ is symmetric, $\maxap(S) = \{w_{e-1}\}$. Also $w_{e-1} - v_n = v_{\wh{e-1}-n}$ for all $v_n \in \Ap(S)$. Thus we have the next three equalities. The last equality follows from Lemma \ref{garcia}.
\end{proof}

Recall from Section 4 that for a semigroup $S$, we always have $\gamma \le \ord(C) \le \tau\le g(a_1) \le r$. By Proposition \ref{someeq}, $\ord(C) = \tau$ if $S$ is symmetric and $g(a_1) = r$ if $S$ is $M$-additive. We can say more if $S$ is $M$-additive symmetric.

\begin{cor}\label{gam=tau} If $S$ is an $M$-additive symmetric semigroup, then $\gamma = \ord(C) = \tau\le g(a_1)=r$.
\end{cor}

\begin{proof} $g(a_1) = r$ by Proposition \ref{someeq}. Thus it suffices to show that $\gamma = \tau$. Indeed
\begin{eqnarray*}
\tau &=& g(f+a_1+1) \text{ by Proposition \ref{cons}}\\
&=& \min_{0\le j\le e-1} \left\{\ord\left(v_{\wh{e-1} + \wh p -\wh \jmath}\right) + \floor{v_{\wh \jmath-\wh p}} \right\} \text{ by Proposition \ref{gu5}}\\
&=& \min_{0\le j\le e-1} \{\gamma_j\} \\
&=& \gamma
\end{eqnarray*}
\end{proof}

The next proposition and its corollary also give the Goto numbers of $S$ in terms of the Ap\'ery set of $S$. This time we consider an $M$-additive symmetric semigroup with the added condition that $\tau=g(a_1)$. This includes the case when $S$ is $M$-symmetric, which is the subject of its corollary.

\begin{prop}\label{maddsym5} Let $S$ be $M$-additive symmetric and $\tau=g(a_1)$. Then
\begin{enumerate}
\item $g(w_i) = \min_{0\le j<i}\, \{\ord(v_{\wh{e-1}+\wh{\imath}-\wh{\jmath}}) + \sfloor{v_{\wh{\jmath}-\wh{\imath}}}\}$ for $1\le i\le e-1$.
\item $g(u) = \ord(w_{e-1}) = \tau$ for all other $u\in S$.
\end{enumerate}
\end{prop}

\begin{proof} (1) follows from Proposition \ref{gu5}, and (2) follows from Propositions \ref{gAp} and \ref{cons}.
\end{proof}

\begin{cor}\label{msym5} Let $S$ be $M$-symmetric. Then
\begin{enumerate}
\item $g(w_i) = \min_{0\le j<i}\, \{\tau + \sfloor{v_{\wh{\jmath}-\wh{\imath}}} - \ord(v_{\wh{\jmath}-\wh{\imath}}) \}$ for $1\le i\le e-1$.
\item $g(u) = \ord(w_{e-1}) = \tau$ for all other $u\in S$.
\end{enumerate}
\end{cor}

\begin{proof} Using Proposition \ref{maddsym5}, we only need to notice that $\ord(v_{\wh{e-1}+\wh{\imath}-\wh{\jmath}}) = \ord(w_{e-1} - v_{\wh{\jmath}-\wh{\imath}}) = \ord(w_{e-1}) - \ord(v_{\wh{\jmath}-\wh{\imath}}) = \tau - \ord(v_{\wh{\jmath}-\wh{\imath}})$.
\end{proof}

Finally we come to the the main results of this section. Theorems \ref{arith} and \ref{sym almost}, regarding symmetric semigroups generated by an arithmetic sequence and symmetric semigroups of almost maximal embedding dimension, make use of Corollary \ref{msym5}; whereas Theorem \ref{maxbed} considers semigroups of maximal embedding dimension and makes use of Lemma \ref{gs}. The last two results, Theorems \ref{ele4} and \ref{ee5}, provide a complete computation of the Goto numbers of the semigroups with multiplicity less than 5, and with multiplicity equal to 5 if the semigroup is symmetric.

\renewcommand{\arraystretch}{1.2}

\begin{thm}\label{arith} Let $S$ be a semigroup generated by an arithmetic sequence such that $q =(e-2)/(\nu-1)$ is an integer. Then
\begin{enumerate}
\item $g(a_2+ka_\nu) = \sfloor{a_2 + (q-k)a_\nu} + k$ where $0\le k\le q$.
\item $g(a_i+ka_\nu) = \sfloor{a_{\nu-i+3} + (q-k+1)a_\nu}+k+1$ where $3\le i\le \nu$ and $0\le k\le q-1$.
\item $g(u) = q+1 = \tau$ for all other $0\ne u\in S$.
\end{enumerate}
\end{thm}

\begin{proof} By Corollary \ref{grg}, $S$ is $M$-symmetric and we can apply Corollary \ref{msym5}. Thus $g(w_i) = \min_{0\le j<i}\, \{\tau + \sfloor{v_{\wh{\jmath}-\wh{\imath}}} - \ord(v_{\wh{\jmath}-\wh{\imath}}) \}$ for $1\le i\le e-1$, and $g(u) = \tau$ for all other $u\in S$. Let $2\le n\le \nu$ and $0\le k\le q$. We have $\Ap(S) = \{0, a_n+ka_\nu\}$ for $2\le n+k(\nu-1) \le e$, and $\ord(a_n+ka_\nu) = k+1$. So $\tau = \ord(w_{e-1}) = \ord(a_2+qa_\nu) = q+1$, giving us (3). For (1) and (2), we need the following claims which we verify at the end of the proof

\begin{claims}\
\begin{enumerate}
\item $v_{\wh{\imath}} = v_{i\wh{1}}$ for $0\le i\le e-1$.
\item $\sfloor{w_i} - \ord(w_i) \le \sfloor{w_{i+1}} - \ord(w_{i+1})$ for $0\le i<e-1$.
\item $w_i = a_n + ka_\nu \iff i=k(\nu-1)+n-1$
\end{enumerate}
\end{claims}

\noindent Using Claim (1), $v_{\wh{\jmath}-\wh{\imath}} = v_{j\wh{1} - i\wh{1}} = v_{(e+j-i)\wh{1}} = w_{e+j-i}$, and by Claim (2) we have $$g(w_i) = \min_{0\le j<i}\, \left\{\tau + \floor{w_{e+j-i}} - \ord(w_{e+j-i}) \right\} = \tau + \floor{w_{e-i}} - \ord(w_{e-i})$$ for $1\le i\le e-1$.
Now let $w_i = a_n+ka_\nu$. By Claim (3) $i = k(\nu-1)+n-1$, and so $e-i = q(\nu-1)+2 -(k(\nu-1)+n-1) = (q-k)(\nu-1) -n+3$. We have two cases. If $n=2$, then $e-i = (q-k)(\nu-1) +1$ and $w_{e-i} = a_2 +(q-k)a_\nu$. If $3\le n\le \nu$, then $e-i = (q-k-1)(\nu-1) +\nu-n+2$ and $w_{e-i} = a_{\nu-n+3} +(q-k-1)a_\nu$. Thus we have
\begin{eqnarray*}
g(a_2+ka_\nu) &=& \tau + \floor{w_{e-i}} - \ord(w_{e-i})\\
&=& \tau + \floor{a_2 +(q-k)a_\nu} - \ord(a_2 +(q-k)a_\nu)\\
&=& q+1 + \floor{a_2 +(q-k)a_\nu} - (q-k+1)\\
&=& \floor{a_2 +(q-k)a_\nu} + k\\
\end{eqnarray*}
\noindent and for $3\le n\le \nu$
\begin{eqnarray*}
g(a_n+ka_\nu) &=& \tau + \floor{w_{e-i}} - \ord(w_{e-i})\\
&=& \tau + \floor{a_{\nu-n+3} +(q-k-1)a_\nu} - \ord(a_{\nu-n+3} +(q-k-1)a_\nu)\\
&=& q+1 + \floor{a_{\nu-n+3} +(q-k-1)a_\nu} - (q-k)\\
&=& \floor{a_{\nu-n+3} +(q-k-1)a_\nu} + k+1\\
\end{eqnarray*}

\noindent This proves (1) and (2) of the proposition.

Now it remains to verify the claims. Claim (3) is clear. For Claim (1), note that there exists a $d$ such the $a_n = e+(n-1)d$ for $2\le n\le \nu$. If $w_i = a_n+ka_\nu$, then
\begin{eqnarray*}
\wh{\imath} &\equiv& e+(n-1)d + k\big(e+(\nu-1)d\big) \mod e \\
&\equiv& (k+1)e + (k(\nu-1)+n-1)d \mod e\\
&\equiv& (k+1)e + (k(\nu-1)+n-1)\wh{1} \mod e\\
&\equiv& (k(\nu-1)+n-1)\wh{1} \mod e\\
&\equiv& i\wh{1} \mod e
\end{eqnarray*}
\noindent So $w_i = v_{\wh{\imath}} = v_{i\wh{1}}$. For Claim (2), let $w_i = a_n + ka_\nu$. We consider two cases. First assume that $w_{i+1} = a_{n+1} + ka_\nu$. Then $\ord(w_i) = \ord(w_{i+1})$, and thus the result follows. For the second case assume that $w_{i+1} = a_{2} + (k+1)a_\nu$. Then we have $w_i = a_{\nu} + ka_\nu$. Thus $\ord(w_i) +1= \ord(w_{i+1})$ and $\sfloor{w_i}+1 \le \sfloor{w_{i+1}}$. So $\sfloor{w_i} - \ord(w_i) = \sfloor{w_i} +1 - \ord(w_i) -1 \le \sfloor{w_{i+1}} - \ord(w_{i+1})$.
\end{proof}

Corollary \ref{com e=2} considers the special case when $S=<a_1,a_2>$. This case was also considered by  \citet[Theorem 5.5]{HS}, and the corollary extends their theorem.

\begin{cor}\label{com e=2} Let $S=<a_1,a_2>$ be a semigroup. Then
\begin{enumerate}
\item $g(ka_2) = a_2 +k -2 - \sfloor{ka_2}$ for $1\le k\le a_1-1$.
\item $g(u) = a_1-1=\tau$ for all other $u\in S$.
\end{enumerate}
\end{cor}

\begin{proof} $S$ is generated by an arithmetic sequence and $q=e-2/\nu-1 = e-2$. So we can apply Theorem \ref{arith}. Thus for $1\le k\le a_1-1$,
\begin{eqnarray*}
g(ka_2) &=& \floor{(a_1-k)a_2} +k-1\\
&=& a_2+k-1 +\floor{-ka_2}\\
&=& a_2+k-1 -\ceil{ka_2}\\
&=& a_2+k-2 -\floor{ka_2}\\
\end{eqnarray*}
This proves (1). Now for any other $0\ne u\in S$, $u\not\in \Ap(S)$. Thus $g(u) = q+1 = a_1-1$, and this proves (2).
\end{proof}

\begin{thm}\label{sym almost} Let $S$ be symmetric and of almost maximal embedding dimension. Then
\begin{enumerate}
\item $g(a_2) = \left\{ \begin{array} {c@{\quad}l} \sfloor{a_2+a_\nu} & \text{if } v_{-\wh{1}} = a_2+a_\nu\\ 1+\sfloor{a_k} & \text{if } v_{-\wh{1}} = a_k\end{array} \right.$
\item $g(a_{i+1}) = 1 + \sfloor{a_k}$ where $a_k=\min\{v_{\wh{\jmath}-\wh{\imath}}\st 0\le j<i\}$ for $2\le i\le \nu-1$
\item $g(a_2+a_\nu) = 1 + \sfloor{a_2}$
\item $g(u) = 2$ for all other $u\in S$.
\end{enumerate}
\end{thm}

\begin{proof} By Corollary \ref{grg}, $S$ is $M$-symmetric and we can apply Corollary \ref{msym5}. We have $\Ap(S)\setminus\{0\} = \{a_2,a_3,\dots,a_\nu,a_i+a_j\}$ where $i+j = \nu+2$. (4) follows from noting that for $u\not\in \Ap(S)$, $g(u) = g(a_1) = \ord(a_2+a_\nu) = 2$. For (1), we note that $g(a_2) = 2+\sfloor{\ds{v}_{-\wh{1}}} - \ord(v_{-\wh{1}})$.

To see (2), let $2\le i\le \nu-1$. Then
$g(a_{i+1}) = \min\{2 + \sfloor{v_{\wh{\jmath}-\wh{\imath}}} - \ord(v_{\wh{\jmath}-\wh{\imath}}) \st 0\le j<i\}$.
If $v_{\wh{\jmath}-\wh{\imath}} = a_k$ for some $k$, then $2 + \sfloor{v_{\wh{\jmath}-\wh{\imath}}} - \ord(v_{\wh{\jmath}-\wh{\imath}}) = 1+\sfloor{a_k}$. Otherwise $v_{\wh{\jmath}-\wh{\imath}} = a_2+a_\nu$ and $2 + \sfloor{v_{\wh{\jmath}-\wh{\imath}}} - \ord(v_{\wh{\jmath}-\wh{\imath}})=\sfloor{a_2+a_\nu}$. Note that for $2\le k\le \nu$, $1+\sfloor{a_k} = \sfloor{a_1+a_k} \le \floor{a_{\nu+2-k} + a_k} = \sfloor{a_2+a_\nu}$. Thus, since $\#\A(a_{i+1}) >1$,
$$g(a_{i+1}) = 1 + \floor{a_k} \text{ where } a_k=\min\{v_{\wh \jmath-\wh \imath}\st 0\le j<i\}$$

\noindent Similarly for (3), we have $g(a_2+a_\nu) = 1 + \sfloor{a_k} \text{ where } a_k=\min\{v_{\wh{\jmath}-\wh{e-1}}\st 0\le j<e-1\} = a_2$.
\end{proof}

\begin{thm}\label{maxbed} Let $S$ have maximal embedding dimension.
\begin{enumerate}
\item $g(a_2) = \sfloor{a_k}$ where $a_k=v_{-\wh{1}}$
\item $g(a_{i+1}) = \sfloor{a_k}$ where $a_k = \min\{v_{\wh{\jmath}-\wh{\imath}} \st 0\le j<i\}$ for $1< i< \nu-1$
\item $g(a_\nu) = \sfloor{a_2}$
\item $g(u)=1$ for all other $u\in S$.
\end{enumerate}
\end{thm}

\begin{proof} $\Ap(S) = \{0,a_2,a_3,\dots,a_\nu\}$. Thus by Lemma \ref{gs}, we have $\sigma(\,\wh{\imath}\,) = \max\{\ord\big(a_k-v_{-\wh{\imath}} + a_1\sfloor{v_{-\wh{\imath}}}\big)\st 2\le k\le \nu\}$. For $1\le i\le e-1$, if $a_k=v_{-\wh \imath}$, then
$\ord\big(a_k-v_{-\wh \imath} + a_1\sfloor{v_{-\wh \imath}}\big) = \sfloor{v_{-\wh \imath}}$. Otherwise $a_k \ne v_{-\wh \imath}$ and
\begin{eqnarray*}
\ord\left(a_k-v_{-\wh \imath} + a_1\floor{v_{-\wh \imath}}\right) &=& \ord\left(v_{\wh{k-1}-\wh{\imath}} + a_1\left(\floor{v_{-\wh \imath}}-t\right)\right) \text{ where } t>0\\
&=& \left\{ \begin{array} {c@{\quad}l} -1 & \text{if } \floor{v_{-\wh \imath}}<t\\ \floor{v_{-\wh \imath}}-t' & \text{otherwise }\end{array}  \right.
\end{eqnarray*}
\noindent where $t'\ge0$. From this we conclude that $\sigma(\,\wh{\imath}\,) = \sfloor{v_{-\wh \imath}}$ if $1\le i\le e-1$.

Now (1), (2), and (3) follow from Lemma \ref{au}, and (4) follows from Proposition \ref{gAp} since for $u\not\in \Ap(S)$, $g(u) = g(a_1) = \ord(a_\nu) = 1$.
\end{proof}

To apply Theorems \ref{sym almost} and \ref{maxbed}, the equations $a_k=v_{-\wh{1}}$ and $a_k = \min\{v_{\wh{\jmath}-\wh{\imath}} \st 0\le j<i\}$ must be solved. The proof of part (3) of Theorem \ref{ele4} illustrates how this can be done. The proofs of the other parts of Theorems \ref{ele4} and \ref{ee5} for which Theorems \ref{sym almost} and \ref{maxbed} apply follow similarly and are omitted. The same idea also works when Proposition \ref{maddsym5} is used for part (2) of Theorem \ref{ee5}.

\begin{thm}\label{ele4} Let $S$ be a semigroup with $e\le 4$. Then one of the following assertions holds.
\begin{enumerate}
\item[$e=2$]
\item $S=<2,a_2>$
\begin{itemize}
\item[] $g(a_2) = \sfloor[2]{a_2}$
\item[] $g(u) = 1$ for all other $u\in S$
\end{itemize}
\item[]
\item[$e=3$]
\item $S=<3,a_2>$
\begin{itemize}
\item[] $g(a_2) = a_2 -\sfloor[3]{a_2}-1$
\item[] $g(2a_2) = a_2 -\sfloor[3]{2a_2}$
\item[] $g(u) = 2$ for all other $u\in S$
\end{itemize}
\item $S=<3,a_2,a_3>$
\begin{itemize}
\item[] $g(a_2) = \sfloor[3]{a_3}$
\item[] $g(a_3) = \sfloor[3]{a_2}$
\item[] $g(u) = 1$ for all other $u\in S$
\end{itemize}
\item[]
\item[e=4]
\item $S=<4,a_2>$
\begin{itemize}
\item[] $g(a_2) = a_2 -\sfloor[4]{a_2}-1$
\item[] $g(2a_2) = a_2 -\sfloor[4]{2a_2}$
\item[] $g(3a_2) = a_2 -\sfloor[4]{3a_2}+1$
\item[] $g(u) = 3$ for all other $u\in S$
\end{itemize}
\item $S=<4,a_2,a_3>$
\begin{enumerate}
\item $S$ is symmetric.
\begin{itemize}
\item[] $g(a_2) = \left\{ \begin{array} {c@{\quad}l} 1+\sfloor[4]{a_2} & \text{if } a_2 \text{ is even}\\ \sfloor[4]{a_2+a_3} & \text{ otherwise} \end{array} \right.$
\item[] $g(a_3) = 1+\sfloor[4]{a_3}$
\item[] $g(a_2+a_3) = 1+\sfloor[4]{a_2}$
\item[] $g(u) = 2$ for all other $u \in S$
\end{itemize}
\item $S$ is not symmetric and $S\ne <4,5,7>$.
\begin{itemize}
\item[] $g(a_2) = \sfloor[4]{a_3}$
\item[] $g(a_3) = 1+\sfloor[4]{a_2}$
\item[] $g(2a_2) = \left\{ \begin{array} {c@{\quad}l} \sfloor[4]{2a_2} & \text{if } 2a_2<a_3\\ \sfloor[4]{a_3} & \text{if } a_3=a_2+2\\ 1+\sfloor[4]{a_2} & \text{otherwise } \end{array} \right.$
\item[] $g(u) = 2$ for all other $u\in S$
\end{itemize}
\item $S = <4,5,7>$.
\begin{itemize}
\item[] $g(4) = 2$
\item[] $g(7) = 2$
\item[] $g(u) = 1$ for all other $u \in S$
\end{itemize}
\end{enumerate}
\item $S=<4,a_2,a_3,a_4>$
\begin{itemize}
\item[] $g(a_2) = \left\{ \begin{array} {c@{\quad}l} \sfloor[4]{a_2} & \text{if } a_2 \text{ is even}\\ \sfloor[4]{a_3} & \text{if } a_4 \text{ is even}\\\sfloor[4]{a_4} & \text{if } a_3 \text{ is even} \end{array} \right.$
\item[] $g(a_3) = \left\{ \begin{array} {c@{\quad}l} \sfloor[4]{a_2} & \text{if } a_4 \text{ is even}\\ \sfloor[4]{a_3} & \text{otherwise } \end{array} \right.$
\item[] $g(a_4) = \sfloor[4]{a_2}$
\item[] $g(u) = 1$ for all other $u\in S$
\end{itemize}
\end{enumerate}
\end{thm}

\begin{proof} (1) and (6) follow from Theorem \ref{maxbed}; (2) and (4) follow from Corollary \ref{com e=2}; (5a) follows from Theorem \ref{sym almost}; and (5c) was worked out in Example \ref{457}.

(3) Theorem \ref{maxbed} can be applied. We have the following information
\begin{enumerate}
\item[] $0 = v_{\wh 0} = v_{0\wh 1}$
\item[] $a_2 = v_{\wh 1}$
\item[] $a_3 = v_{\wh 2} = v_{2\wh 1}$
\end{enumerate}

\noindent giving us $v_{-\wh{1}}=v_{2\wh 1}=a_3$, and $a_2 = \min\{v_{\wh{\jmath}-\wh{2}} \st 0\le j<2\}$ since $v_{-\wh{2}}=v_{\wh 1}=a_2$ and $v_{\wh{1}-\wh{2}} =v_{\wh{1}-2\wh 1} = v_{-\wh{1}}=v_{2\wh 1}=a_3$.

For (5b) we have $\Ap(S) = \{0,a_2,2a_2,a_3\}$ or $\Ap(S) = \{0,a_2,a_3,2a_2\}$. In either case $\maxap\{a_3,2a_2\}$. Let $w_j=a_3$ and $w_k=2a_2$ where $\{j,k\} = \{2,3\}$. We have $\sigma(\,\wh{0}\,) = \ord{2a_2} = 2$. The other elements of the Goto set can be determined using Lemma \ref{gs}. First notice that $v_{-\wh{1}} = a_3$, $v_{-\wh{\jmath}} = a_2$, and $v_{-\wh{k}} = 2a_2$. Thus\\

$\sigma(\,\wh{1}\,) = \sfloor{a_3}$ since
\begin{eqnarray*}
\ord\left(a_3 - a_3+a_1\floor{a_3}\right)\;\;\! &=& \floor{a_3}\text{, and}\\
\ord\left(2a_2- a_3+a_1\floor{a_3}\right) &=& \ord\left(a_3+a_1\left(\floor{a_3}-t\right)\right) \text{ where } t>0\\
&=& \left\{ \begin{array} {c@{\quad}l} \floor{a_3}-t' \text{ where } t'\ge 0 & \text{if } \floor{a_3}\ge t'\\ -1 & \text{otherwise } \end{array} \right.
\end{eqnarray*}

$\sigma(\,\wh{\jmath}\,) = 1+ \sfloor{a_2}$ since
\begin{eqnarray*}
\ord\left(2a_2 - a_2+a_1\floor{a_2}\right) &=& 1+\floor{a_2}\text{, and}\\
\ord\left(a_3- a_2+a_1\floor{a_2}\right)\;\; &=& \ord\left(2a_2+a_1\left(\floor{a_2}-t\right)\right) \text{ where } t>0\\
&=& \left\{ \begin{array} {c@{\quad}l} \floor{a_2}-t' \text{ where } t'\ge -1 & \text{if } \floor{a_2}\ge t'\\ -1 & \text{otherwise } \end{array} \right.
\end{eqnarray*}

$\sigma(\,\wh{k}\,) = \sfloor{2a_2}$ since
\begin{eqnarray*}
\ord\left(2a_2 - 2a_2+a_1\floor{2a_2}\right) &=& \floor{2a_2}\text{, and}\\
\ord\left(a_3- 2a_2+a_1\floor{2a_2}\right)\;\; &=& \ord\left(a_2+a_1\left(\floor{2a_2}-t\right)\right) \text{ where } t>0\\
&=& \left\{ \begin{array} {c@{\quad}l} \floor{2a_2}-t' \text{ where } t'\ge 0 & \text{if } \floor{2a_2}\ge t'\\ -1 & \text{otherwise } \end{array} \right.
\end{eqnarray*}

Clearly $\sigma(\,\wh{0}\,)$ is less than or equal to both $\sigma(\,\wh{\jmath}\,)$ and $\sigma(\,\wh{k}\,)$. Moreover, by the assumption that $S\ne<4,5,7>$, $\sigma(\,\wh{0}\,)$ is less than or equal to $\sigma(\,\wh{1}\,)$. Thus $g(a_1) = \sigma(\,\wh{0}\,) = \tau$, and for $u\not\in \Ap(S)$, $g(u) = g(a_1) = \sigma(\,\wh{0}\,) = 2$.

Now to compute the Goto numbers for the nonzero elements of $\Ap(S)$, we consider some cases. Using Lemma \ref{au} we have
\begin{enumerate}
\item If $2a_2 < a_3$
\begin{enumerate}
\item $\A(2a_2) = \{\wh{1},\wh{k}\}$
\item $\A(a_3) = \{\wh{1},\wh{\jmath},\wh{k}\}$
\item $\sigma(\,\wh{0}\,)\le \sigma(\,\wh{\jmath}\,)\le \sigma(\,\wh{k}\,)\le \sigma(\,\wh{1}\,)$
\end{enumerate}
\item If $a_3< a_1+a_2$
\begin{enumerate}
\item $\A(a_3) = \{\wh{\jmath},\wh{k}\}$
\item $\A(2a_2) = \{\wh{1},\wh{\jmath},\wh{k}\}$
\item $\sigma(\,\wh{0}\,)\le \sigma(\,\wh{1}\,)\le \sigma(\,\wh{\jmath}\,)\le \sigma(\,\wh{k}\,)$
\end{enumerate}
\item Otherwise
\begin{enumerate}
\item $\A(a_3) = \{\wh{\jmath},\wh{k}\}$
\item $\A(2a_2) = \{\wh{1},\wh{\jmath},\wh{k}\}$
\item $\sigma(\,\wh{0}\,)\le \sigma(\,\wh{\jmath}\,)\le \sigma(\,\wh{1}\,)\le \sigma(\,\wh{k}\,)$
\end{enumerate}
\end{enumerate}

\noindent Now (5b) follows from Corollary \ref{cmain} once we observe that if $a_3 < a_1+a_2$, then $a_3 = a_2+2$.
\end{proof}

\begin{thm}\label{ee5} Let $S$ be a symmetric semigroup with $e=5$. Then one of the following assertions holds.
\begin{enumerate}
\item $S=<5,a_2>$
\begin{itemize}
\item[] $g(a_2) = a_2 - \sfloor[5]{a_2}-1$
\item[] $g(2a_2) = a_2 - \sfloor[5]{2a_2}$
\item[] $g(3a_2) = a_2 - \sfloor[5]{3a_2}+1$
\item[] $g(4a_2) = a_2 - \sfloor[5]{4a_2}+2$
\item[] $g(u) = 4$ for all other $u\in S$
\end{itemize}
\item[]
\item $S=<5,a_2,a_3>\ne <5,6,9>$
\begin{itemize}
\item[] $g(a_2) = 1+\sfloor[5]{a_3}$
\item[] $g(a_3) = 2+\sfloor[5]{a_2}$
\item[] $g(2a_2) = 1+\sfloor[5]{a_3}$
\item[] $g(3a_2) = 2+\sfloor[5]{a_2}$
\item[] $g(u) = 3$ for all other $u\in S$
\end{itemize}
\item[]
\item $S=<5,6,9>$
\begin{itemize}
\item[] $g(5) = 3$
\item[] $g(9) = 3$
\item[] $g(14) = 3$
\item[] $g(u) = 2$ for all other $u\in S$
\end{itemize}
\item[]
\item $S=<5,a_2,a_3,a_4>$
\begin{itemize}
\item[] $g(a_2) = \left\{ \begin{array} {c@{\quad}l} 1+\sfloor[5]{a_3} & \text{if } 5\st a_2 +a_3 \\ \sfloor[5]{2a_3} & \text{otherwise} \end{array} \right.$
\item[] $g(a_3) = \left\{ \begin{array} {c@{\quad}l} 1+\sfloor[5]{a_2} & \text{if } 5\st a_2 +a_3\\ 1+\sfloor[5]{a_4} & \text{otherwise} \end{array} \right.$
\item[] $g(a_4) = 1+\sfloor[5]{a_3}$
\item[] $g(2a_3) = 1+\sfloor[5]{a_2}$
\item[] $g(u) = 2$ for all other $u\in S$
\end{itemize}
\end{enumerate}
\end{thm}

\begin{proof} (1) follows from Corollary \ref{com e=2}; (2) follows from Proposition \ref{maddsym5}, which can be applied due to Proposition \ref{grcm} and Theorem \ref{except}; and (4) follows from Theorem \ref{sym almost}.

(3) $S=<5,6,9> = \{0,5,6,9,10,11,12,14,\rightarrow\}$ where the $\rightarrow$ indicates that all integers greater than 14 are in $S$. Then $T = \{13\}$ and $\A(S)=\{1,2,3,4,5\}$. We have $\sigma(1) = \ord(14) = 2$, $\sigma(2) = \ord(15) = 3$, $\sigma(3) = \ord(16) = 3$, and $\sigma(4) = \ord(17) = 3$. Now the Goto vector of $S$ is $gv=(2,3,3,3,3)$ and we see that for $0\ne u \in S$, if $1\in \A(u)$, then $g(u) = 2$ and otherwise $g(u) = 3$. Thus the result follows.
\end{proof}


\section*{Acknowledgement}

The author would like to thank William Heinzer for his guidance and help during the writing of this paper. He would also like to thank YiHuang Shen for many useful conversations and examples. The computer algebra system Singular was used for all computer calculations.

\end{document}